\documentclass[11pt,reqno]{amsart}
\usepackage{amsmath,amsfonts,amssymb,amsthm,amscd,comment,euscript}
\usepackage[all]{xy}
\usepackage{graphicx}
\usepackage{mathptmx}
\usepackage{enumerate} 
\usepackage[colorlinks=true]{hyperref} 
\usepackage[usenames,dvipsnames]{xcolor}
\usepackage{enumitem}

\xymatrixcolsep{1.9pc}                          
\xymatrixrowsep{1.9pc}
\newdir{ >}{{}*!/-5\pt/\dir{>}}                  

 \calclayout

\swapnumbers

\theoremstyle{plain}
\newtheorem{lem}{Lemma}[section]
\newtheorem{cor}[lem]{Corollary}
\newtheorem{prop}[lem]{Proposition}
\newtheorem{thm}[lem]{Theorem}

\theoremstyle{definition}

\newtheorem{exs}[lem]{Examples}
\newtheorem{rem}[lem]{Remark}
\newtheorem{dfn}[lem]{Definition}

\renewcommand{\phi}{\varphi}
\renewcommand{\leq}{\leqslant}
\renewcommand{\geq}{\geqslant}
\renewcommand{\epsilon}{\varepsilon}

\renewcommand{\kappa}{\varkappa}

\DeclareMathOperator{\spec}{Spec}\DeclareMathOperator{\tor}{tor}
 
\DeclareMathOperator{\Gr}{Gr} 
\DeclareMathOperator{\diag}{diag} \DeclareMathOperator{\charr}{char}
\DeclareMathOperator{\hocolim}{hocolim}

 \DeclareMathOperator{\mot}{mot}

\DeclareMathOperator{\Hom}{Hom} 
 \DeclareMathOperator{\id}{id}
\DeclareMathOperator{\Fr}{Fr} \DeclareMathOperator{\Tot}{Tot}
 
\DeclareMathOperator{\colim}{colim}
 \DeclareMathOperator{\Ab}{Ab}
 
 \DeclareMathOperator{\kr}{Ker}
 \DeclareMathOperator{\im}{Im}
 \DeclareMathOperator{\coker}{Coker}
\DeclareMathOperator{\nis}{\mathsf{nis}}
 \DeclareMathOperator{\Tor}{Tor}
 
 \DeclareMathOperator{\Mod}{Mod}

\newcommand{\Shv}{\mathbf{Shv}}
\newcommand{\Sm}{\mathbf{Sm}}

\newcommand{\lra}[1]{\bl{#1}\longrightarrow\relax}
\newcommand{\bl}[1]{\buildrel #1\over}
\newcommand{\cc}{\mathcal}
\newcommand{\bb}{\mathbb}

\newcommand{\op}{{\textrm{\rm op}}}

\newcommand{\wt}{\widetilde}

\newcommand{\gmp}{\bb G}
\newcommand{\gmpn}{\bb G^{\wedge n}}
\newcommand{\uhom}{\underline{\Hom}}

\newcommand{\M}{\mathcal{M}}

\newcommand{\eff}{\mathsf{eff}}

\newcommand{\pt}{\mathsf{pt}}

\newcommand{\PP}{\mathbb{P}}
\newcommand{\w}{\wedge}

\begin{document}

\footskip30pt


\title{Computing framed motives}
\author{Grigory Garkusha}
\address{Department of Mathematics, Swansea University, Fabian Way, Swansea SA1 8EN, UK}
\email{g.garkusha@swansea.ac.uk}
\urladdr{\url{https://ggarkusha.github.io}}

\begin{abstract}
We develop a spectral sequence framework for computing framed motives associated with motivic Thom spectra. Our 
principal tool is a motivic Atiyah--Hirzebruch spectral sequence relating framed motives to framed motivic cohomology. 
This provides a systematic computational approach to framed motives analogous to the role of Atiyah--Hirzebruch-type 
spectral sequences in stable homotopy theory. As an application, after inverting a finite set of primes, we compute the 
bigraded homotopy sheaves of motivic Thom spectra in terms of framed motivic cohomology. We further analyze the 
symmetric group actions inherent in framed correspondences and introduce a theory of torsion framed motivic cohomology, 
yielding new computational descriptions of framed motivic cohomology groups. Finally, these constructions lead to a 
category of permutation-free framed correspondences from which we reconstruct rational stable motivic homotopy theory over any field.
\end{abstract}

\thanks{Explore interactive visualisations on framed correspondences and motives at \url{https://ggarkusha.github.io/math-atlas/}}


\maketitle
\thispagestyle{empty}
\pagestyle{plain}


\section{Introduction}\label{introduction}

In stable motivic homotopy theory, the introduction of framed correspondences by Voevodsky~\cite{Voe2}
has fundamentally transformed our understanding of motivic spectra. Building on Voevodsky’s foundational notes~\cite{Voe2},
framed motives of algebraic varieties were introduced and studied in~\cite{GP1} in order to suggest a new approach to stable
motivic homotopy theory envisioned by Voevodsky in the 2000s.
Framed motives have emerged as one of the most powerful computational and conceptual tools in stable motivic homotopy theory. 
In particular, they furnish explicit fibrant models for motivic~spectra.

Though the machinery of framed motives leads to explicit geometric models of stable motivic homotopy theory, effective methods 
for computing them have remained limited. The difficulty is not unique to framed motives. In stable homotopy theory, 
spectra are rarely amenable to direct computation. Instead, one typically replaces a spectrum by a suitable filtration or resolution 
and extracts its homotopy groups through the associated spectral sequence. This philosophy underlies virtually all major 
computational methods in the subject, including the Atiyah--Hirzebruch, Adams, Adams--Novikov, and slice spectral sequences. 
These spectral sequences provide successive approximations to stable homotopy groups while relating them to more 
accessible algebraic or geometric invariants.

From this perspective, a natural question arises: can one develop an analogous computational framework 
for framed motives? Although framed motives have become one of the principal geometric models for motivic spectra, 
no systematic spectral sequence machinery has previously been available for computing their stable homotopy groups.
The principal objective of this paper is to fill this gap. We introduce in Section~\ref{mathssmass} natural filtrations of framed motives and 
construct the associated spectral sequences converging to their stable homotopy groups. These spectral sequences 
provide a new computational framework for framed motives, analogous to the role played by the classical spectral 
sequences of stable homotopy theory. Besides establishing their convergence and functoriality, we identify their 
$E^2$-terms and demonstrate that they can be used to obtain explicit computations of framed motives and related motivic invariants.

For example, among the most fundamental invariants in motivic homotopy theory are the stable motivic homotopy groups
$\pi_{*,\star}^{\bb A^1}(S_T)$ of the motivic 
sphere spectrum $S_T$. Their computation is at least as challenging as that of the 
stable homotopy groups of the classical sphere spectrum $\bb S$ (see, e.g.,~\cite{Lev}), 
while simultaneously encoding rich arithmetic information unavailable in the topological setting.
In Theorem~\ref{mathssthom} we construct a motivic Atiyah--Hirzebruch spectral sequence for  a motivic Thom spectrum $E$.
In particular, this spectral sequence provides a precise link between the stable motivic homotopy groups of $S_T$, 
the stable homotopy groups of $\bb S$, and framed motivic cohomology,
thereby relating topological and motivic invariants through a common computational framework.

At present, our understanding of the stable motivic homotopy groups $\pi_{*,\star}^{\bb A^1}(S_T)$ 
remains rather limited. We refer the reader to~\cite{RS} for a survey of the most recent developments.  
These groups encode highly subtle arithmetic and geometric information, much of which continues to elude existing computational techniques.
One of the principal motivations for developing a motivic Atiyah--Hirzebruch spectral sequence is to obtain new tools for studying such phenomena.
As an application of Theorems~\ref{ovechkin} and~\ref{mathssthom}, we show that, after inverting a finite 
set of primes determined by the stable homotopy groups of spheres, 
the bigraded stable motivic homotopy sheaves of a motivic Thom spectrum $E$ can be identified with 
$E$-framed motivic cohomology groups.

The motivic Atiyah--Hirzebruch spectral sequence, constructed in Section~\ref{mathssmass}, establishes its 
relationship with framed motivic cohomology. It is therefore important to compute the corresponding cohomology groups. 
To this end, we develop in Section~\ref{nikishin} additional 
tools for such computations by exploiting the symmetric group actions that are intrinsic to the theory of framed correspondences.

Namely, the presence of permutation symmetries in framed correspondences suggests the existence of a richer algebraic 
structure than is visible at the level of ordinary framed motivic cohomology. We show that these symmetries give rise 
to torsion phenomena controlled by symmetric groups and lead naturally to a theory of torsion framed motivic 
cohomology introduced in Section~\ref{nikishin}. This viewpoint allows us to isolate the 
contribution of permutation actions and obtain new descriptions of 
framed motivic cohomology groups (see Theorem~\ref{rakushka}).
By Theorem~\ref{shell} permutations are rationally invisible --- they contribute only torsion, captured precisely by the
$\Tor_p$-groups of Theorem~\ref{rakushka}.

The results of Section~\ref{nikishin} also reveal that a substantial part of the homotopy theoretic information carried by 
framed correspondences is encoded in structures that remain visible after eliminating permutation phenomena. 
In particular, the theory of permutation-free framed correspondences developed in Section~\ref{nikishin} suggests 
that symmetric group actions play a secondary role from the perspective of rational stable motivic homotopy theory. 
Namely, our final result (see Theorem~\ref{reconstr}) shows that permutation-free framed correspondences suffice to recover rational stable motivic homotopy 
theory. More precisely, we use the category of rational permutation-free framed correspondences and investigate 
the associated category of motives. We prove that the resulting motivic category is equivalent to the rational stable 
motivic homotopy category. This provides a new model for $SH(k)_{\bb Q}$ over any field in which we can sacrifice the 
combinatorial complexity arising from permutation actions, while the essential homotopical information 
is preserved. The reconstruction theorem may therefore be viewed as a conceptual explanation for the computational results obtained 
in the previous sections and as evidence that permutation-free framed correspondences capture the core 
structure of rational motivic phenomena.
It is worth mentioning that the new model for $SH(k)_{\bb Q}$ is independent of~\cite{GG19}, where $SH(k)_{\bb Q}$ was recovered out of 
Milnor--Witt correspondences in the sense of~\cite{CF}. Also, we do not assume here that the field
$k$ is perfect of characteristic different from two in contrast with the reconstruction theorem of~\cite{GG19}.

\subsubsection*{Notation.}
Throughout the paper we employ the following notation.\label{ntn}
\vspace{0.08in}

\begin{tabular}{l|l}
$k$ and $\pt$ & a field of exponential characteristic $p$ and $\mathrm{Spec}(k)$, respectively\\
$\Sm_k$ & the category of smooth separated $k$-schemes of finite type \\
$\Shv_{\bullet}(\Sm_k)$ or $\Shv_{\bullet}$ & the closed symmetric monoidal category of pointed Nisnevich sheaves \\
$\M=\Delta^{\op}\Shv_{\bullet}$ & the category of pointed motivic spaces\\
$Sp_{S^1}$ & the category of $S^1$-spectra of pointed simplicial sets\\
$Sp_{S^1}(k)$ & the category of $S^1$-spectra of pointed motivic spaces
\end{tabular}

We also denote by $\Sm'_k$ the category of essentially smooth $k$-schemes. Following~\cite{GrD},
an essentially smooth $k$-scheme is a Noetherian $k$-scheme
$X$ that can be expressed as the inverse limit of a filtered inverse system
$(X_i)_{i\in I}$, where each transition morphism $X_i \to X_j$ being
an \'{e}tale affine morphism between smooth $k$-schemes.
We shall freely use the main result of~\cite{GP4} --- complemented by~\cite{DP} in
characteristic 2 and by~\cite[Theorem~B.32]{DKO} over finite fields --- which asserts that if $k$ is perfect then 
for any $\bb A^1$-invariant quasi-stable $\bb ZF_*(k)$-presheaf of Abelian
groups $\mathcal F$, the associated Nisnevich sheaf $\mathcal
F_{\nis}$ is strictly $\bb A^1$-invariant. If $k$ is imperfect, we will also use the strict $\bb A^1$-invariance 
theorem with $1/p$-coefficients of~\cite[Theorem~A]{GG26}.
If $\mathcal C$ is a category
cotensored over the category of pointed motivic spaces $\mathcal M$, we write
$\uhom(A,C)\in\mathcal C$ for the cotensor object associated with $A\in\mathcal M$ and $C\in\mathcal C$, unless stated otherwise.
For $\cc X\in\cc M$, the Suslin complex is the motivic space $C_*(\cc X):=\diag(\uhom_{\cc M}(-\times\Delta_k^\bullet,\cc X))$, where
$\Delta^n_k=\spec(k[x_0,\ldots,x_n]/(x_0+\cdots+x_n-1))$ denotes the algebraic $n$-simplex.
Throughout the paper, we work with the flasque local and motivic model structures on motivic spaces, and on 
$S^1$- and $T$-spectra of motivic spaces, in the sense of~\cite{Is}.


\section{Framed correspondences and their $\Gamma$-spaces}\label{hutson}

The main objective of this section is to analyze the homotopy-theoretic structure of Segal $\Gamma$-spaces
associated with framed correspondences. They will be necessary
for the construction of the motivic Atiyah--Hirzebruch spectral sequence. We study $S^1$-spectra arising from 
framed correspondences and obtain explicit descriptions of their stable homotopy types. These results constitute 
the technical input for the spectral sequence constructions of Section~\ref{mathssmass}.

The starting point is Voevodsky’s description of morphisms between certain Nisnevich sheaves in terms 
of geometric data. This fundamental identification allows one to pass freely between sheaf-theoretic 
and geometric formulations and serves as the principal computational tool in the machinery of framed motives~\cite{GP1}.

The proof of Voevodsky's Lemma can be found in~\cite[Section~3]{GP1}.

\begin{lem}[Voevodsky]\label{Voevlemma}
For $X,Y\in\Sm_k$ and a closed subset $X'$ of $X$ and open subset $V$ of $Y$ the set
   $\Hom_{\Shv_{\bullet}}(X/X',Y/V)$
is in a natural bijection with the set of equivalence classes of
triples $\Phi=(U,Z,\phi)$, where $Z$ is a closed subset of $X$ disjoint
with $X'$, $U$ is an \'etale neighborhood of $Z$ in $X$ and
$\phi\colon U\to Y$ is a regular map such that $\phi^{-1}(Y-V)=Z$.
Two triples $\Phi=(U,Z,\phi)$, $\Phi'=(U',Z',\phi')$ are {\it
equivalent\/} if $Z=Z'$ and $\phi,\phi'$ coincide on some common
etale neighbourhood of $Z$ in $X$. 
\end{lem}

Under the bijection of Lemma~\ref{Voevlemma} we shall often identify elements in $\Hom_{\Shv_{\bullet}}(X/X',Y/V)$ with 
associated triples $(U,Z,\phi)$. Denote by $\hom_{\Shv_{\bullet}}(X/X',Y/V)$ the subset of $\Hom_{\Shv_{\bullet}}(X/X',Y/V)$
of those triples $(U,Z,\phi)$  such that $Z$ is connected. It is pointed at the triple $0:=(U,Z,\phi)$ with $U=\emptyset$.

For a pointed set $(K,*)$, the motivic Thom space $Y/V\wedge K$ equals $(\sqcup_{K\setminus *}Y)/(\sqcup_{K\setminus *}V)$.
The association
   $$K\in\Gamma^{\op}\longmapsto\Hom_{\Shv_{\bullet}}(X/X',Y/V\wedge K)$$
defines a $\Gamma$-space denoted by $\Gamma\Hom_{\Shv_{\bullet}}(X/X',Y/V)$. The $\Gamma$-space
$\Gamma\hom_{\Shv_{\bullet}}(X/X',Y/V)$ is defined in a similar fashion. 

Similarly to~\cite[Appendix~B]{GNP} one has
   $$\Gamma\hom_{\Shv_{\bullet}}(X/X',Y/V)=\bigvee_{\Phi\in\hom_{\Shv_{\bullet}}(X/X',Y/V)\setminus 0}\Gamma\hom_{\Shv_{\bullet}}(X/X',Y/V)_{\Phi}.$$
Here $\Gamma\hom_{\Shv_{\bullet}}(X/X',Y/V)_{\Phi}$  is the $\Gamma$-subspace
$K\in\Gamma^{\op}\mapsto\{(U,Z,\phi,g)\mid\Phi=(U,Z,\phi),g:U\to K\setminus*\}\sqcup\{0\}$. Note that it is isomorphic to $\Gamma(1_+,-)$.

Thus the following fact is true.

\begin{lem}\label{dorofeev}
There is a natural isomorphism of $\Gamma$-spaces
   $$\Gamma\hom_{\Shv_{\bullet}}(X/X',Y/V)\cong\bigvee_{\Phi\in\hom_{\Shv_{\bullet}}(X/X',Y/V)\setminus 0}\Gamma(1_+,-).$$
\end{lem}

In turn --- similarly to~\cite[Appendix~B]{GNP},
  $$\Gamma\Hom_{\Shv_{\bullet}}(X/X',Y/V)=\bigcup_{\Psi\in\Hom_{\Shv_{\bullet}}(X/X',Y/V)\setminus 0}\Gamma\Hom_{\Shv_{\bullet}}(X/X',Y/V)_{\Psi}.$$
Here $\Gamma\Hom_{\Shv_{\bullet}}(X/X',Y/V)_{\Psi}$ is defined as follows.
Given $\Phi=(U,Z,\phi)$ and $\Phi'=(U',Z',\phi')$ in $\Hom_{\Shv_{\bullet}}(X/X',Y/V)$, we write $\Phi'\leq \Phi$ if there is a closed subset $Z''$
in $X$ such that $Z=Z'\sqcup Z''$ and
   $$(U',Z';\phi')=(U-Z'',Z';\phi|_{U-Z''}) \in\Hom_{\Shv_{\bullet}}(X/X',Y/V).$$
By definition,
   $$\Gamma\Hom_{\Shv_{\bullet}}(X/X',Y/V)_{\Psi}(K):=\{(U,Z,\phi,g:U\to K\setminus*)\in\Hom_{\Shv_{\bullet}}(X/X',Y/V\wedge K)\mid(U,Z,\phi)\leq\Psi\}.$$
   
We set, 
   $$\Gamma'\Hom_{\Shv_{\bullet}}(X/X',Y/V)_{\Psi}:=\Gamma\Hom_{\Shv_{\bullet}}(X/X',Y/V)_{\Psi}\cap \Gamma\hom_{\Shv_{\bullet}}(X/X',Y/V).$$
If $\Psi=(U,Z,\psi)$ then the $\Gamma$-space $\Gamma\Hom_{\Shv_{\bullet}}(X/X',Y/V)_{\Psi}$
is isomorphic to $\Gamma(cc(Z)_+,-)$, where $cc(Z)$ is the finite set of connected components of $Z$. 
By construction, $\Gamma'\Hom_{\Shv_{\bullet}}(X/X',Y/V)_{\Psi}$
is isomorphic to the $\Gamma$-space $\bigvee_{cc(Z)}\Gamma(1_+,-)$.

Recall from~\cite[Section~4]{BF} that any $\Gamma$-space $A$ can be extended to a functor $A:Sp_{S^1}\to Sp_{S^1}$, which is defined by the rule
   $$A(M)=\int^{n_+\in\Gamma^{\op}}M^{\times n}\wedge A(n_+),\quad M\in Sp_{S^1}.$$
Note that $A$ preserves filtered colimits of spectra. Also, $(\colim_IA_i)(M)=\colim_IA_i(M)$ for any filtered set of indices $I$. 

Let $\Hom_{\Shv_{\bullet}}^{M}(X/X',Y/V)$ and 
$\hom_{\Shv_{\bullet}}^{M}(X/X',Y/V)$ denote the corresponding evaluations of the $\Gamma$-spaces
$\Gamma\Hom_{\Shv_{\bullet}}(X/X',Y/V)$ and $\Gamma\hom_{\Shv_{\bullet}}(X/X',Y/V)$ at $M\in Sp_{S^1}$.

\begin{thm}\label{protas}
The inclusion of $\Gamma$-spaces
   $$\Gamma\hom_{\Shv_{\bullet}}(X/X',Y/V)\hookrightarrow\Gamma\Hom_{\Shv_{\bullet}}(X/X',Y/V)$$
induces a stable equivalence of spectra 
   $$\hom_{\Shv_{\bullet}}^{M}(X/X',Y/V)\xrightarrow{\sim}\Hom_{\Shv_{\bullet}}^{M}(X/X',Y/V)$$
for any $M\in Sp_{S^1}$. Moreover, for any $n\in\bb Z$ the $n$-th stable homotopy group 
$\pi_n(\Hom_{\Shv_{\bullet}}^{M}(X/X',Y/V))$ of the $S^1$-spectrum
$\Hom_{\Shv_{\bullet}}^{M}(X/X',Y/V)$ is isomorphic to $\bb Z\hom_{\Shv_{\bullet}}(X/X',Y/V)\otimes\pi_n(M)$, where
$\pi_n(M)$ is the $n$-th stable homotopy group of $M$.
\end{thm}

\begin{proof}
Set $\mathcal A=\hom_{\Shv_{\bullet}}(X/X',Y/V)\setminus\{0\}$.
For $\Psi\in\Hom_{\Shv_{\bullet}}(X/X',Y/V)$, set
\[
S_\Psi=\{\Phi\in\mathcal A\mid\Phi\leq\Psi\}.
\]
Thus $S_\Psi$ is the finite set of connected support
components of $\Psi$. The decomposition of a correspondence
into its connected support components is unique. Therefore
the identifications
\[
\Gamma\Hom_{\Shv_{\bullet}}(X/X',Y/V)_\Psi
\cong\Gamma((S_\Psi)_+,-)
\]
are compatible with restriction to subsets of $S_\Psi$.
In particular, the intersection of the pieces associated
with $\Psi$ and $\Psi'$ is the piece indexed by
$S_\Psi\cap S_{\Psi'}$.

For a finite subset $T\subseteq\mathcal A$, write
\[
P_T(M)=\prod_{\Phi\in T}M,
\qquad
W_T(M)=\bigvee_{\Phi\in T}M.
\]
Regard these as coordinate subspectra of
$\prod_{\Phi\in\mathcal A}M$ with the remaining coordinates
equal to the basepoint. For $T=\emptyset$, both spectra are
the zero spectrum. We have
\[
P_T(M)\cap P_{T'}(M)=P_{T\cap T'}(M),
\qquad
W_T(M)\cap W_{T'}(M)=W_{T\cap T'}(M).
\]

The coend defining evaluation at $M$ preserves colimits
in the $\Gamma$-space variable. Moreover,
\[
\Gamma(T_+,-)(M)\cong P_T(M),
\qquad
\left(\bigvee_{\Phi\in T}\Gamma(1_+,-)\right)(M)
\cong W_T(M).
\]
Together with the intersection description above, these
give compatible identifications
\[
\Hom_{\Shv_{\bullet}}^M(X/X',Y/V)
\cong\bigcup_\Psi P_{S_\Psi}(M)
\]
and
\[
\hom_{\Shv_{\bullet}}^M(X/X',Y/V)
\cong\bigcup_\Psi W_{S_\Psi}(M)
=\bigvee_{\Phi\in\mathcal A}M.
\]
Here the last equality follows because every
$\Phi\in\mathcal A$ occurs as a correspondence with the
single connected support component $\Phi$.

We verify that the inclusion of these unions is a stable
equivalence. More generally, for any finite family
$T_1,\ldots,T_r$ of finite subsets of $\mathcal A$, the map
\[
\bigcup_{i=1}^r W_{T_i}(M)
\longrightarrow
\bigcup_{i=1}^r P_{T_i}(M)
\]
is a stable equivalence. We argue by induction on $r$.
For $r=1$, this is the canonical stable equivalence from
a finite wedge of copies of $M$ to their product.

Suppose that $r>1$ and that the assertion holds for
families with fewer than $r$ members. Set
\[
W'=\bigcup_{i<r}W_{T_i}(M),
\qquad
P'=\bigcup_{i<r}P_{T_i}(M).
\]
The intersection identities give
\[
W'\cap W_{T_r}(M)
=\bigcup_{i<r}W_{T_i\cap T_r}(M),
\qquad
P'\cap P_{T_r}(M)
=\bigcup_{i<r}P_{T_i\cap T_r}(M).
\]
By induction, the maps between these intersections and
between $W'$ and $P'$ are stable equivalences. The map
$W_{T_r}(M)\to P_{T_r}(M)$ is also a stable equivalence.
The two unions are the pushouts along the displayed
intersections. All the inclusions are levelwise
monomorphisms of pointed simplicial sets, so these
pushout squares are levelwise homotopy pushouts.
This is enough for proving the induction step.

Now take the filtered colimit over finite subfamilies
of the family $\{S_\Psi\}_\Psi$. Stable homotopy groups
of $S^1$-spectra commute with filtered colimits, so the
resulting inclusion
\[
\hom_{\Shv_{\bullet}}^M(X/X',Y/V)
\longrightarrow
\Hom_{\Shv_{\bullet}}^M(X/X',Y/V)
\]
is a stable equivalence.

Finally, stable homotopy groups commute with wedges.
Hence, for every $n\in\mathbb Z$,
\begin{align*}
\pi_n\bigl(\Hom_{\Shv_{\bullet}}^M(X/X',Y/V)\bigr)
&\cong
\pi_n\left(\bigvee_{\Phi\in\mathcal A}M\right)\\
&\cong
\bigoplus_{\Phi\in\mathcal A}\pi_n(M)\cong
\mathbb Z\hom_{\Shv_{\bullet}}(X/X',Y/V)
\otimes_{\mathbb Z}\pi_n(M),
\end{align*}
as was to be shown.
\end{proof}

The proof of the preceding theorem also implies the following result.

\begin{cor}\label{lindgren}
For every $M\in Sp_{S^1}$ the spectrum $\Hom_{\Shv_{\bullet}}^{M}(X/X',Y/V)$ has the stable homotopy type of
$\bigvee_{\Phi\in\hom_{\Shv_{\bullet}}(X/X',Y/V)\setminus 0} M=\hom_{\Shv_{\bullet}}(X/X',Y/V)\wedge M$.
\end{cor}

\begin{cor}\label{wilson}
The functors
   $$M\in Sp_{S^1}\mapsto\Hom_{\Shv_{\bullet}}^{M}(X/X',Y/V)\in Sp_{S^1}$$
and
   $$M\in Sp_{S^1}\mapsto\hom_{\Shv_{\bullet}}^{M}(X/X',Y/V)\in Sp_{S^1}$$
respect stable equivalences and stable homotopy (co)fiber sequences.
\end{cor}

\begin{proof}
This immediately follows from Corollary~\ref{lindgren}.
\end{proof}

\begin{dfn}\label{geomconn}
Given $A,B\in\Sm_k$, a closed subset $A'$ of $A$ and open subset $C$ of $B$, we say that a morphism
$\tau:A/A'\to B/C$ of pointed Nisnevich sheaves is {\it geometrically connected\/} if 
the support $Z$ of the associated triple $\Phi=(U,Z,\phi)$ (see Lemma~\ref{Voevlemma}) is geometrically connected over $k$.
For any $n\geq 1$ let $A_n\subset A^{\times n}$ be a closed subset which is the union of all subsets of the form
$A\times\cdots\times A'\times\cdots\times A$. Set $A_0=\emptyset$. Likewise, let $C_n\subset B^{\times n}$ be an open 
subset which is the union of all subsets of the form
$B\times\cdots\times C\times\cdots\times B$. Set $C_0=\emptyset$.

Given $n\geq 0$, $X,Y\in\Sm_k$, a closed subset $X'$ of $X$ and open subset $V$ of $Y$, a {\it $\tau$-framed correspondence
of level $n$ from $X/X'$ to $Y/V$\/} is a morphism of pointed Nisnevich sheaves
   $$X\times A^{\times n}/(X'\times A^{\times n}\cup X\times A_n)\longrightarrow Y\times B^{\times n}/(V\times B^{\times n}\cup Y\times C_n).$$
Denote the set of $\tau$-framed correspondences of level $n$ from $X/X'$ to $Y/V$ by
$\Fr_n^\tau(X/X',Y/V)$. It is pointed at the empty correspondence. 
If $\tau$ is given by the canonical geometrically connected morphism $\sigma:\bb P^{\wedge 1}\to T$,
where $\bb P^{\wedge 1}=\bb P^1/\infty$, $T=\bb A^1/\bb G_m$, then we drop $\tau$ from notation and just write
$\Fr_n(X/X',Y/V)$. Recall that $\sigma$ is associated to the triple $(\{0\},\bb A^1, t)$ if we use Lemma~\ref{Voevlemma}.

Let $F_n^\tau(X/X',Y/V)$ be the pointed set of those $\tau$-framed correspondences $\Phi=(U,Z,\phi)\in\Fr_n^\tau(X/X',Y/V)$ 
for which $Z$ is connected. There is a natural map
   $$-\wedge\tau:\Fr_n^\tau(X/X',Y/V)\to \Fr_{n+1}^\tau(X/X',Y/V).$$
Since $\tau$ is geometrically connected, it follows from~\cite[Lemma~33.7.4]{Stack} that the restriction of this map 
to $F_n^\tau(X/X',Y/V)$ lands in $F_{n+1}^\tau(X/X',Y/V)$.

The set of {\it stable $\tau$-framed correspondences from $X/X'$ to $Y/V$\/} is the pointed set
   $$\Fr^\tau(X/X',Y/V)=\colim(\Fr_0^\tau(X/X',Y/V)\xrightarrow{-\wedge\tau}\Fr_{1}^\tau(X/X',Y/V)\xrightarrow{-\wedge\tau}\cdots).$$
The set $F^\tau(X/X',Y/V)$ is defined in a similar fashion.

Finally, we define {\it linear $\tau$-framed correspondences of level $n$ from $X/X'$ to $Y/V$\/} 
as the reduced free Abelian group, denoted by $\bb ZF_n^\tau(X/X',Y/V)$,
associated to the pointed set $F_n^\tau(X/X',Y/V)$. In turn,
{\it linear stable $\tau$-framed correspondences from $X/X'$ to $Y/V$\/} are the elements of the Abelian group
$\bb ZF^\tau(X/X',Y/V)=\colim_n\bb ZF_n^\tau(X/X',Y/V)$.
\end{dfn}

For any finite pointed set $(K,*)$, we will write $\Fr^\tau(X/X',Y/V\otimes K)$ to denote the set of stable $\tau$-framed
correspondences from $X/X'$ to $(\sqcup_{K\setminus *}Y)/(\sqcup_{K\setminus *}V)$.
As above, we have two $\Gamma$-spaces
   $$\Gamma\Fr^\tau(X/X',Y/V):K\in\Gamma^{\op}\longmapsto\Fr^\tau(X/X',Y/V\otimes K)$$
and
   $$\Gamma F^\tau(X/X',Y/V):K\in\Gamma^{\op}\longmapsto F^\tau(X/X',Y/V\otimes K)$$
We define $\Gamma$-spaces $\Gamma\Fr^\tau_n(X/X',Y/V)$ and $\Gamma F^\tau_n(X/X',Y/V)$ in a similar fashion.
By construction, 
   $$\Gamma\Fr^\tau(X/X',Y/V)=\colim_n\Gamma\Fr_n^\tau(X/X',Y/V)\quad\textrm{and}\quad
       \Gamma F^\tau(X/X',Y/V)=\colim_n\Gamma F^\tau_n(X/X',Y/V).$$
   
If $M\in Sp_{S^1}$ we will write $\Fr^\tau(X/X',Y/V\otimes M)$ and $F^\tau(X/X',Y/V\otimes M)$ to denote the values of
the $\Gamma$-spaces $\Gamma\Fr^\tau(X/X',Y/V)$ and $\Gamma F^\tau(X/X',Y/V)$ at $M$.

\begin{thm}\label{carbery}
The inclusion of $\Gamma$-spaces
   $$\Gamma F^\tau(X/X',Y/V)\hookrightarrow\Gamma\Fr^\tau(X/X',Y/V)$$
induces a stable equivalence of spectra 
   $$F^\tau(X/X',Y/V\otimes M)\xrightarrow{\sim}\Fr^\tau(X/X',Y/V\otimes M)$$
for any $M\in Sp_{S^1}$. Moreover, for any $n\in\bb Z$ the $n$-th stable homotopy group 
$\pi_n(\Fr^\tau(X/X',Y/V\otimes M))$ of the $S^1$-spectrum
$\Fr^\tau(X/X',Y/V\otimes M)$ is isomorphic to $\bb ZF^\tau(X/X',Y/V)\otimes\pi_n(M)$, where
$\pi_n(M)$ is the $n$-th stable homotopy group of $M$.
\end{thm}

\begin{proof}
By Theorem~\ref{protas} the map of spectra
  $$F^\tau_n(X/X',Y/V\otimes M)\xrightarrow{}\Fr^\tau_n(X/X',Y/V\otimes M)$$
is a stable equivalence for any $M\in Sp_{S^1}$ and any $n\geq 0$. As filtered colimits of stable equivalences
are stable equivalences, so is the map of spectra
  $$F^\tau(X/X',Y/V\otimes M)\xrightarrow{}\Fr^\tau(X/X',Y/V\otimes M).$$
  
It also follows from Theorem~\ref{protas} that $\pi_n(\Fr^\tau_m(X/X',Y/V\otimes M))$ 
is isomorphic to $\bb ZF^\tau_m(X/X',Y/V)\otimes\pi_n(M)$ for any $m\geq 0$ and $n\in\bb Z$.
As colimits of isomorphisms are isomorphisms, we see that the group
$\pi_n(\Fr^\tau(X/X',Y/V\otimes M))$ 
is isomorphic to $\bb ZF^\tau(X/X',Y/V)\otimes\pi_n(M)$.
\end{proof}

The proof of the preceding theorem together with Corollary~\ref{lindgren} also implies the following result.

\begin{cor}\label{strome}
For every $M\in Sp_{S^1}$ the spectrum $\Fr^\tau(X/X',Y/V\otimes M)$ has stable homotopy type of
$\bigvee_{\Phi\in F^\tau(X/X',Y/V)\setminus 0} M=F^\tau(X/X',Y/V)\wedge M$.
\end{cor}

\begin{cor}\label{roy}
The functors
   $$M\in Sp_{S^1}\mapsto\Fr^\tau(X/X',Y/V\otimes M)\in Sp_{S^1}$$
and
   $$M\in Sp_{S^1}\mapsto F^\tau(X/X',Y/V\otimes M)\in Sp_{S^1}$$
respect stable equivalences and stable homotopy (co)fiber sequences.
Moreover, if $M$ is a connective spectrum, so are $\Fr^\tau(X/X',Y/V\otimes M)$
and $F^\tau(X/X',Y/V\otimes M)$.
\end{cor}

\begin{proof}
This immediately follows from Corollary~\ref{strome}.
\end{proof}

Let $SH$ be the stable homotopy category of $S^1$-spectra. An {\it Adams resolution\/} of a spectrum $M\in SH$ is a sequence in $SH$
   \begin{equation}\label{adams}
    \xymatrix{\cdots\ar[r]&M^2\ar[d]\ar[r]&M^1\ar[d]\ar[r]&M^0=M\ar[d]\\
                                                           &L^2&L^1&L^0}
   \end{equation}                                                        
such that:
\begin{enumerate}
\item for each $i\geq 0$, $M^{i+1}\to M^i\to L^i\lra{+}$ is a triangle in $SH$;
\item each $L^i$ is an $S^1$-spectrum whose spaces are simplicial Abelian groups.
\end{enumerate}

\begin{exs}\label{adamsresex}
(1) Let $SH_{\geq 0}$ be the full subcategory of $SH$ of $(-1)$-connected spectra. 
Every $M\in SH_{\geq 0}$ fits in a canonical tower (see, e.g.,~\cite[Section~II.8]{Sch})
   \begin{equation}\label{tower}
    \cdots\to M^2\to M^1\to M^0=M
   \end{equation}
with $M^n$ being $(n-1)$-connected. Each layer $L^n=cofib(M^{n+1}\to M^n)$
is isomorphic to the $n$-shifted Eilenberg--MacLane spectrum $EM(\pi_n(M))[n]$ of the Abelian group $\pi_n(M)$.
We will also refer to this Adams resolution~\eqref{tower} as the {\it Postnikov tower for $M$}.

(2) A typical Adams resolution of a spectrum $M\in SH_{\geq 0}$ is defined as follows. Let $\bb Z(M)$ be the linearization
spectrum for $M$. In more detail, if $M=(M_0,M_1,\ldots)$ then $\bb Z(M):=(\bb Z(M_0),\bb Z(M_1),\ldots)$, where each
$\bb Z(M_i)$ is the reduced simplicial free Abelian group associated to $M_i$. Set $M^1:=fib(M\to\bb Z(M))$. In this way $M$
has the following Adams resolution:
   $$\xymatrix{\cdots\ar[r]&M^2\ar[d]\ar[r]&M^1\ar[d]\ar[r]&M^0=M\ar[d]\\
                                                           &\bb Z(M^2)&\bb Z(M^1)&\bb Z(M^0)}$$
Note that each $M^n$ is $(n-1)$-connected, and hence so is $\bb Z(M^n)$.                                                
\end{exs}

The following result is a general principle for producing spectral sequences associated to framed correspondences. 
It will be of great utility for computing framed motives of algebraic varieties.

\begin{thm}\label{oshie}
If $M\in SH$ has an Adams resolution of the form~\eqref{adams}, then 
the spectrum $\Fr^\tau(X/X',Y/V\otimes M)$ has an Adams resolution of the form
   $$ \xymatrix{\cdots\ar[r]&\Fr^\tau(X/X',Y/V\otimes M^2)\ar[d]\ar[r]&\Fr^\tau(X/X',Y/V\otimes M^1)\ar[d]\ar[r]&\Fr^\tau(X/X',Y/V\otimes M)\ar[d]\\
                                                           &\bb ZF^\tau(X/X',Y/V)\otimes L^2&\bb ZF^\tau(X/X',Y/V)\otimes L^1&\bb ZF^\tau(X/X',Y/V)\otimes L^0}$$
\end{thm}

\begin{proof}
It follows from Corollary~\ref{roy} that the functors
   $$M\in Sp_{S^1}\mapsto\Fr^\tau(X/X',Y/V\otimes M)\in Sp_{S^1}$$
is lifted to a triangulated functor
   $$M\in SH\mapsto\Fr^\tau(X/X',Y/V\otimes M)\in SH.$$
Consider the image of the tower~\eqref{adams} under this functor. By Corollary~\ref{strome} each
$\Fr^\tau(X/X',Y/V\otimes L^i)$ is isomorphic to $F^\tau(X/X',Y/V)\w L^i$ in $SH$. Since $L^i$ is a spectrum formed by 
simplicial Abelian groups, the natural map of spectra $F^\tau(X/X',Y/V)\w L^i\to\bb ZF^\tau(X/X',Y/V)\otimes L^i$
is a stable equivalence. We see that $\Fr^\tau(X/X',Y/V\otimes L^i)$ is isomorphic to $\bb ZF^\tau(X/X',Y/V)\otimes L^i$ in $SH$.
\end{proof}

\section{MATHSS and MASS for framed motives}\label{mathssmass}

In the previous section, we established the homotopical properties of spectra arising from framed correspondences 
and obtained explicit descriptions of their stable homotopy types. We now apply these results to develop computational 
tools for framed motives. The principal objective of this section is the construction of a motivic Atiyah--Hirzebruch spectral sequence 
that relates framed motives associated with motivic Thom spectra to framed motivic cohomology. This spectral sequence 
provides a systematic mechanism for extracting homotopy-theoretic information from cohomological invariants
and serves as one of the main computational devices for framed motives.

The results of Section~\ref{hutson} 
supply the necessary descriptions of the associated homotopy types, allowing us to identify the $E^2$-page explicitly 
and establish convergence properties. As an application, we obtain spectral sequences computing the bigraded homotopy 
sheaves of broad classes of motivic Thom spectra (including the motivic sphere spectrum $S_T$ or the algebraic cobordism spectrum $MGL$)
and derive explicit cohomological descriptions of their 
homotopy invariants after suitable localization. It is worth mentioning that our spectral 
sequences are different from any kind of slice spectral sequences
as we work in the $S^1$-direction rather than in the $\gmp$-direction. More precisely, we apply
framed motives to towers of $S^1$-spectra.

Let us start with preparations.
Let $T^n=\bb A^n/(\bb A^n-0)$. For any $\mathcal X\in\cc M$ and any
$T$-spectrum $E$ denote by $\Fr_n^E(\mathcal X)=\uhom(\PP^{\w
n},\mathcal X\w E_n)$, $n\geqslant 0$. Also, set $\Fr^E(\mathcal
X)=\colim_n(\uhom(\PP^{\w n},\mathcal X\w E_n))$. If $\mathcal X=X_+$, $X\in
\Sm_k$, then we shall write $\Fr^E(X)$ dropping $+$ from notation.
We also drop $E$ from notation if $E$ is the motivic sphere spectrum $S_T=(S^0,T,T^2,\ldots)$.

By~\cite[Lemma~9.1]{GN} $\Fr^E(\mathcal X)$ is functorial in $\mathcal X$ and $E$. If $E$ is
a directed colimit of $T$-spectra $\colim_kE_k$, then
$\Fr^E(\mathcal X)=\colim_k\Fr^{E_k}(\mathcal X)$. In particular,
$\Fr^E(\mathcal X)=\colim_k\Fr^{L_kE}(\mathcal X)$, where $L_kE=(E_0,\ldots,E_{k-1},E_k,E_k\wedge T,E_k\w T^2,\ldots)$ is
the $k$-th layer of $E$.

\begin{dfn}
(1) Given a $T$-spectrum $E$, the assignment 
   $$C_*\Fr^E(\mathcal X):X\in\Sm_k,K\in\Gamma^{\op}\mapsto C_*\Fr^E(X,\mathcal X\w K),$$ 
is plainly a $\Gamma$-space. If $N\in Sp_{S^1}$ the {\it $E$-framed motive
$M_E(\mathcal X,N)$ of $\mathcal X$ with coefficients in $N$\/} is the
value of the $\Gamma$-space $C_*\Fr^E(\mathcal X)$ at $N$. By definition, $M_E(\mathcal X,N)\in Sp_{S^1}(k)$.
If $E=S_T$ and $N=\bb S=(S^0,S^1,\ldots)$, then
$M_E(\mathcal X,N)$ is the framed motive $M_{fr}(\mathcal X)$ of
$\mathcal X$ in the sense of~\cite{GP1}.

(2) Following~\cite{GN} a $T$-spectrum $E$ is called a {\it Thom spectrum}
if every space $E_n$ has the form
\[E_n=\colim_i E_{n,i}, \text{  }E_{n,i}=V_{n,i}/(V_{n,i}-Z_{n,i}),\]
where $V_{n,i}\to V_{n,i+1}$ is a directed sequence of smooth
varieties, $Z_{n,i}\to Z_{n,i+1}$ is a directed system of smooth
closed subschemes in $V_{n,i}.$ We say that a Thom spectrum $E$ {\it
has the bounding constant $d$\/} if $d$ is the minimal integer such
that codimension of $Z_{n,i}$ in $V_{n,i}$ is strictly greater than
$n-d$ for all $i,n$. If $E$ is also symmetric then it is said to be
a {\it spectrum with contractible alternating group action}, if for
any $n$ and any even permutation $\tau\in \Sigma_n$ there is an
$\bb A^1$-homotopy $E_n\to \uhom(\bb A^1,E_n)$ between the action of
$\tau$ and the identity map. In other words, $E$ neglects the action
of even permutations up to $\bb A^1$-homotopy. The most interesting
examples of such symmetric Thom spectra, all of which have the
bounding constant $d=1$, are given by the spectra $MGL$, $MSL$ or
$MSp$ (the latter two are regarded as $T^2$-symmetric spectra for
which the above definitions remain the same).

(3) Given a Thom $T$-spectrum $E$, $U\in \Sm_k$ and $Y=X/(X-Z)$, where
$X\in \Sm_k$ and $Z$ is a closed subset in $X$, denote by
$\bb ZF_n^E(U,Y)$ the free Abelian group generated by the elements of
$\Fr_n^E(U,Y)=\uhom(\PP^{\w n},Y\w E_n)$ with connected support
(recall that the elements of $\Fr_n^E(U,Y)$ have an explicit
geometric description using Voevodsky's Lemma~\ref{Voevlemma}). We
also set $\bb ZF^E(U,Y):=\colim_n\bb ZF_n^E(U,Y)$, where the colimit maps
are defined in the same fashion as those of $\Fr^E(U,Y)$.

The assignment $K\mapsto C_*\bb ZF^E(U,Y\w K)$ is plainly a
$\Gamma$-space. Let $N\in Sp_{S^1}$. The {\it linear $E$-framed motive $LM_E(Y,N)$ of
$Y$ with coefficients in $N$\/} is the value of the $\Gamma$-space $C_*\bb ZF^E(Y)$ at $N$. 
If $E=S_T$ and $N=\bb S$ then $LM_E(Y,N)$ is the linear framed motive
$LM_{fr}(Y)$ of $Y$ in the sense of~\cite{GP1}. Note that the
presheaves of stable homotopy groups $\pi_*(LM_E(Y))$ are computed
as the presheaves of homology groups of the complex $C_*\bb ZF^E(Y)$
(we freely use the Dold--Kan correspondence here). It follows from Lemma~\ref{dorofeev} that $LM_E(Y,N)$
can be computed as $C_*\bb ZF^E(Y)\otimes\bb Z(N)$, where $\bb Z(N)$ is the 
$S^1$-spectrum obtained from $N$ by taking the reduced free Abelian group in each 
degree.

If $A$ is an Abelian group, we will write $LM_E(Y,A)$ and $\bb ZF^E(Y,A)$ to denote
$LM_E(Y)\otimes A$ and $\bb ZF^E(Y)\otimes A$ respectively.
\end{dfn}

\begin{lem}\label{lapierre}
The following statements are true:

$(1)$ the geometric realization of simplicial stable equivalences of $S^1$-spectra is a stable equivalence;

$(2)$ the geometric realization of simplicial degreewise connective $S^1$-spectra is connective;

$(3)$ if $F_\bullet \to E_\bullet \to B_\bullet$ is a sequence of simplicial $S^1$-spectra
such that in each degree $n$, the sequence
$F_n \to E_n \to B_n$ is a homotopy fibre sequence of spectra (i.e. the composition of the maps is trivial
and the associated map from $F_n$ to the homotopy fiber of $E_n\to B_n$ is a stable equivalence), then the geometric realization
$|F_\bullet| \to |E_\bullet| \to |B_\bullet|$ is a homotopy fibre sequence of spectra. 
\end{lem}

\begin{proof}
(1) Let $f_\bullet: X_\bullet\to Y_\bullet$ be a map of simplicial spectra such that each $f_n:X_n\to Y_n$, $n\geq 0$,
is a stable equivalence. Without loss of generality we may assume that $X_n,Y_n$ are cofibrant spectra. Indeed, consider the
level model structure on spectra and
consider the induced model structure on simplicial spectra, in which weak equivalences and fibrations are defined degreewise~\cite[Theorem~11.6.1]{Hir}.
If $\alpha:X^c\to X$ is a cofibrant replacement of the simplicial spectrum $X$ in this model structure, then $\alpha_n$
is a levelwise weak equivalence in each degree $n$ and $X^c_n$ is a cofibrant spectrum by~\cite[Proposition~11.6.3]{Hir}
(we tacitly use here the fact that cofibrant spectra in the level and stable model structures coincide). Applying
the cofibrant replacement functor to $f_\bullet$, we get a map of simplicial spectra
$f^c_\bullet: X^c_\bullet\to Y^c_\bullet$ with $X_n^c,Y_n^c$ cofibrant. As $X_n^c\to X_n$ and $Y_n^c\to Y_n$
are levelwise weak equivalences, the maps of spectra $|X^c_\bullet|\to |X_\bullet|$ and $|Y^c_\bullet|\to |Y_\bullet|$
are levelwise weak equivalences. Therefore $|f_\bullet|$ is a stable equivalence if and only if so is $|f_\bullet^c|$.

So we assume that $f_\bullet$ is a map between simplicial spectra which are cofibrant in each degree.
By~\cite[Corollary~18.7.5]{Hir} $|f_\bullet|: |X_\bullet|\to |Y_\bullet|$ is levelwise weakly equivalent to 
   $$\hocolim_{\Delta^\op} f_\bullet: \hocolim_{\Delta^\op} X_\bullet\to\hocolim_{\Delta^\op} Y_\bullet$$
The latter map is a stable equivalence by~\cite[Corollary~18.5.3]{Hir}.

(2) Let $X_\bullet$ be a simplicial degreewise connective spectrum. By~\cite[Proposition~18.15]{Dug} there is a spectral sequence
   $$E^1_{p,q}=\pi_p(X_q)\Longrightarrow\pi_{p+q}|X_\bullet|$$
where the differentials have the form $d^r:E^r_{p,q}\to E^r_{p+r-1,q-r}$. As $\pi_{<0}(X_q)=0$ for all $q$, our statement follows.
The statement can also be verified by inspecting negative stable homotopy groups for the tower that induces the spectral sequence
--- see the proof of~\cite[Proposition~18.15]{Dug}.

(3) The proof of the first statement shows that the composite map $F_\bullet^c\lra{\sim}F_\bullet\to E_\bullet$ can be factored
as $F_\bullet^c\hookrightarrow E_\bullet^c\lra{\sim}E_\bullet$, in which the left arrow is a cofibration between cofibrant
simplicial spectra. The left arrow is also a cofibration in each simplicial degree by~\cite[Proposition~11.6.3]{Hir}. 
Set $B_\bullet^c:=E^c_\bullet/F^c_\bullet$. As homotopy fibre sequences of spectra coincide with homotopy cofibre ones,
the induced map $B^c_\bullet\to B_\bullet$ is degreewise a stable equivalence due to our assumption on $B_\bullet$.
It follows from the first statement that $|F_\bullet| \to |E_\bullet| \to |B_\bullet|$ is a homotopy fibre sequence of spectra
if $|F^c_\bullet| \hookrightarrow |E^c_\bullet| \to |B^c_\bullet|$ is a homotopy cofibre sequence of spectra. But the latter sequence is
a homotopy cofibre sequence in the level model structure of spectra as $B_\bullet^c$ is a pushout of the cofibration
$F_\bullet^c\hookrightarrow E_\bullet^c$, and hence $|B_\bullet^c|=\diag(B_\bullet^c)$ is a pushout of the level
cofibration $\diag(F^c_\bullet)\hookrightarrow\diag(E^c_\bullet)$.
\end{proof}

One of the approaches to constructing the Atiyah--Hirzebruch spectral sequence (ATHSS) 
associated with a generalized (Eilenberg--Steenrod) cohomology theory $E$ is via the Postnikov
tower applied to $E$ --- see~\cite{Mau}. In the motivic world, we can apply the same strategy to
constructing the motivic Atiyah--Hirzebruch spectral sequence (MATHSS).

\begin{thm}[MATHSS for framed motives]\label{ovechkin}
Suppose $N$ is a connective $S^1$-spectrum and $E$ is a motivic Thom $T$-spectrum. 
Let $X,U\in\Sm_k$, $Z$ be closed and smooth in $X$ and $Y=X/(X-Z)$.

$(1)$ The Postnikov tower~\eqref{tower} yields a tower in $SH$ of the form
   \begin{equation}\label{mottower} 
    \xymatrix{\cdots\ar[r]&M_E(Y,N^2)(U)\ar[d]\ar[r]&M_E(Y,N^1)(U)\ar[d]\ar[r]&M_E(Y,N)(U)\ar[d]\\
                                     &LM_E(Y,\pi_2(N))(U)[2]&LM_E(Y,\pi_1(N))(U)[1]&LM_E(Y,\pi_0(N))(U)}
   \end{equation}
with each $M_E(Y,N^\ell)(U)$ being $(\ell-1)$-connected.

$(2)$ If the field $k$ is perfect, the tower~\eqref{mottower} yields a tower of $\bb A^1$-local motivic $S^1$-spectra in $SH_{S^1}(k)$
   \begin{equation}\label{shtower} 
    \xymatrix{\cdots\ar[r]&M_E(Y,N^2)\ar[d]\ar[r]&M_E(Y,N^1)\ar[d]\ar[r]&M_E(Y,N)\ar[d]\\
                                     &LM_E(Y,\pi_2(N))_{\nis}[2]&LM_E(Y,\pi_1(N))_{\nis}[1]&LM_E(Y,\pi_0(N))_{\nis}}
   \end{equation}
with each $M_E(Y,N^\ell)$ locally $(\ell-1)$-connected. If $k$ is imperfect, the same is true with $1/p$-coefficients.

$(3)$ If $k$ is perfect, the tower~\eqref{shtower} yields a strongly convergent spectral sequence
   $$E^2_{p,q}=H^{-p}_{\nis}(U,C_*\bb ZF^E(Y,\pi_q(N))_{\nis})\Longrightarrow\pi_{p+q}(M_E(Y,N)^f(U)),$$
where $M_E(Y,N)^f$ is a stable local fibrant replacement of $M_E(Y,N)$. If $k$ is imperfect, the same is true with $1/p$-coefficients.
\end{thm}

\begin{proof}
(1) If $E=S_T$ then our statement follows from Theorem~\ref{oshie}, Lemma~\ref{lapierre} and Corollary~\ref{roy}.
If $E$ is a motivic Thom $T$-spectrum, so is the $T$-spectrum $Y\wedge E$. For this 
reason we may assume $Y=\pt$. It follows from~\cite[Corollary~9.3]{GN} that $M_E(\pt,N)=\colim_k M_{L_kE}(\pt,N)$, where $L_kE$ is the $k$-th layer
of $E$. Also, \cite[Lemma~9.4]{GN} implies $M_{L_kE}(\pt,N)=\uhom(\PP^{\w k},M_{fr}(E_k,N))$
for any $k\geqslant 0$. As $\PP^{\w k}$ is of the form $X/X'$, we can apply results of Section~\ref{hutson} to
$\uhom(\PP^{\w k},\Fr(E_k))=\Fr(-\times\PP^{\w k},E_k)$. Our statement for $M_{L_kE}(\pt,N)$
follows from Theorem~\ref{oshie}, Lemma~\ref{lapierre} and Corollary~\ref{roy}. Since filtered colimits respect stable
equivalences and (co)fiber sequences of spectra, our statement follows for $M_E(\pt,N)$ as well.

(2) As above $M_E(Y,N)=\colim_k\uhom(\PP^{\w k},M_{fr}(E_k,N))$.
As $k$ is perfect, each motivic spectrum $\uhom(\PP^{\w k},M_{fr}(E_k,N))$
has strictly $\bb A^1$-invariant sheaves of stable homotopy groups, and hence so does $M_E(Y,N)$.
By~\cite[Theorem~6.2.7]{Mor} its local fibrant replacement is motivically fibrant.
We see that the motivic spectra of the tower are $\bb A^1$-local.
Our statement now follows from the first one. If $k$ is imperfect, the same arguments apply with $1/p$-coefficients
if we use~\cite[Theorem~A]{GG26}.

(3) By the second statement each $M_E(Y,N^q)$ is $\bb A^1$-local, hence
$M_E(Y,N^q)^f$ is motivically fibrant by~\cite[Theorem~6.2.7]{Mor}. If $U\in\Sm_k$ is of dimension $d$, then
$M_E(Y,N^q)^f(U)$ is $(q-d-1)$-connected
in $Sp_{S^1}$ by~\cite[Lemma~4.3.1]{Mor} as $M_E(Y,N^q)$ is locally $(q-1)$-connected. By~\cite[Corollary~6.1.1]{FS} the tower~\eqref{shtower}
produces a strongly convergent spectral sequence
   $$E^2_{p,q}=\pi_{p+q}(LM_E(Y,\pi_q(N))_{\nis}^f(U)[q])\Longrightarrow\pi_{p+q}(M_E(Y,N)^f(U)).$$
It remains to observe that $\pi_{p+q}(LM_E(Y,\pi_q(N))_{\nis}^f(U)[q])=H^{-p}_{\nis}(U,C_*\bb ZF^E(Y,\pi_q(N))_{\nis})$.
\end{proof}

\begin{dfn}\label{haapsalu}
Let $E$ be a motivic $T$-spectrum. Let $\gmp$ denote the cone of the $1$-section 
$\mathrm{Spec}(k)\to\bb G_{m} $ in $\Delta^{\op}\Fr_0(k)$.
For every integer $q\geq 0$ the motivic complex $\bb Z^E(q)$ is defined
as the following (bounded above) cochain complex of sheaves with framed transfers:
   $$\bb Z^E(q)=C_*\bb ZF^E(\bb G^{\w q})_{\nis}[-q]$$
(the shift is cohomological).
If $A$ is any other Abelian group then $A^E(q) =C_*\bb ZF^E(\bb G^{\w q},A)_{\nis}[-q]$ is another cochain complex of
presheaves with framed transfers. If $q<0$ we set 
   $$A^E(q):=\uhom(\bb G^{\w |q|},A^E(0))[-q].$$
   
The {\it $E$-framed motivic cohomology groups $H^{p,q}_E(X,A)$ with coefficients in $A$} are defined to
be the hypercohomology of the motivic complexes $A^E(q)$ with respect to the Nisnevich
topology:
   $$H^{p,q}_E(X,A):=H_{\nis}^p(X,A^E(q)).$$
If $E=S_T$ we will write $A_{fr}(q)$ for $A^E(q)$ and $H^{p,q}_{fr}(X,A):=H_{\nis}^p(X,A_{fr}(q)).$
\end{dfn}

For $p,n\in\bb Z$ let $\pi^{\bb A^1}_{p,n}(E)$ be the Nisnevich sheaf on $\Sm_k$ associated to the presheaf
   \begin{equation}\label{bigraded}
    U\longmapsto \pi_{p,n}(E)(U)=SH(k)(U_{+}\wedge S^{p-n}\wedge\gmpn, E).
   \end{equation}
Recall that $E$ is connective if $\pi^{\bb A^1}_{p,n}(E)=0$ for all $p<n$.

The following theorem is an application of Theorem~\ref{ovechkin}.

\begin{thm}[MATHSS for motivic Thom spectra]\label{mathssthom}
Suppose $k$ is perfect and $E$ is a symmetric Thom $T$-spectrum with
the bounding constant $d\leq 1$ and contractible alternating group action.
There is a strongly convergent spectral sequence
   $$E^2_{p,q}=H^{-p-n,-n}_{E}(U,\pi_q^s)\Longrightarrow\pi_{p+q+n,n}(E)(U),\quad n\in\bb Z,$$
where $\pi^s_q$ is the $q$-th stable homotopy group of the classical sphere spectrum $\bb S$.
In particular, there is a strongly convergent spectral sequence of Nisnevich sheaves
   $$E^2_{p,q}= H^{-p-n}((\pi_q^s)^E(-n))\Longrightarrow\pi^{\bb A^1}_{p+q+n,n}(E),\quad n\in\bb Z.$$
The same statement is true with $1/p$-coefficients whenever $k$ is imperfect.
\end{thm}

\begin{proof}
By~\cite[Theorem~9.13]{GN} (see~\cite[Theorems~5.5 and~6.3]{GG26} for imperfect fields) the $(S^1,\gmp)$-bispectrum
\[M_E^{\mathbb G}(\pt)^f:=(M_E(\pt)^f,M_E(\gmp)^f,M_E(\gmp^{\w 2})^f,\ldots)\]
is motivically fibrant and represents the $T$-spectrum $E$ in
the category of bispectra, where ``$f$'' refers to stable local
fibrant replacements of $S^1$-spectra. Our theorem now follows from Theorem~\ref{ovechkin}.
\end{proof}

If $\operatorname{char}(k)\ne 2$ and $p$ is the
exponential characteristic of $k$, put
$\Lambda=\mathbb Z[1/p]$ and $E_\Lambda=E[1/p]$.
By~\cite[Theorem~5.5 and Remark~2.2]{RSO} for every
$n\in\mathbb Z$ there is an exact sequence of Nisnevich sheaves
\[
0\longrightarrow K^M_{2-n}\otimes_{\mathbb Z}\Lambda/24
\longrightarrow
\pi_{n+1,n}^{\mathbb A^1}((S_T)_\Lambda)
\longrightarrow
\pi_{n+1,n}^{\mathbb A^1}\bigl(f_0(\mathbf{KQ}_\Lambda)\bigr).
\]
Here $f_0(\mathbf{KQ}_\Lambda)$ is the effective cover of
$\Lambda$-localized hermitian $K$-theory in the slice filtration
of $SH(k)$. The rightmost map is surjective for $n\geq -4$.

The preceding theorem yields another exact sequence for $\pi_{n+1,n}^{\bb A^1}(E)$. More precisely, the following statement is true.

\begin{cor}\label{sandin}
Under the assumptions of Theorem~\ref{mathssthom} there is an exact sequence 
   $$H^{-n}(\bb Z^E(-n))/2\to\pi_{n+1,n}^{\bb A^1}(E)\to H^{-1-n}(\bb Z^E(-n))\to 0,\quad n\in\bb Z,$$
of Nisnevich sheaves. If $\charr(k)\ne 2$ then there is an exact sequence of Abelian groups
   $$K^{MW}_{-n}(k)/2\to\pi_{n+1,n}^{\bb A^1}(S_T)(k)\to H^{-1-n}(\bb Z_{fr}(-n))(k)\to 0,\quad n\in\bb Z.$$
\end{cor}

\begin{proof}
The first exact sequence immediately follows from Theorem~\ref{mathssthom} if we recall that $\pi_1^s(\bb S)=\bb Z/2\bb Z$.
The proof of~\cite[Corollary~11.3]{GP1} shows that
   $$\pi_{-n,-n}^{\bb A^1}(S_T)(k)=H^{n}(\bb Z_{fr}(n))(k),\quad n\in\bb Z.$$
By Morel's theorem~\cite{Mor1} the left hand side Abelian group is $K^{MW}_n(k)$. The exactness of the second sequence is now obvious.
\end{proof}

By the Linearization Theorem of~\cite{GNP} and~\cite[Lemma~9.7]{GN}, for $n\in\bb Z$ the map
$\lambda_{-n}:M_E(\gmp^{\wedge(-n)})\to LM_E(\gmp^{\wedge(-n)})$, $n\leq 0$, ($\lambda_n:=\uhom(\gmp^{n},\lambda_0)$ for $n>0$)
 induces isomorphisms of Nisnevich sheaves
   $$\pi_{n+*,n}^{\bb A^1}(E)_{\bb Q}\cong H^{-*-n}(\bb Q^E(-n)).$$

As an application of Theorem~\ref{mathssthom} we have the following refinement 
of these isomorphisms showing which primes are enough to invert
in order to get similar isomorphisms in each range $[1,\ldots,\ell]$. Recall that
$\pi_{t,s}^{\bb A^1}(E)=0$ for $t<s$ and $\pi_{n,n}^{\bb A^1}(E)\cong H^{-n}(\bb Z^E(-n))$, $n\in\bb Z$.

\begin{thm}\label{mathssapp}
Suppose $k$ is perfect and $E$ is a symmetric Thom $T$-spectrum with
the bounding constant $d\leq 1$ and contractible alternating group action.
For any $\ell\geq 1$ let $\{p_1,\ldots,p_{t(\ell)}\}$ be the set of all prime divisors
of orders of the finite Abelian groups $\pi_1^s,\ldots,\pi_{\ell}^s$. Then, for every $n\in\bb Z$, the linearization morphism
$\lambda_{-n}$ induces a natural isomorphism of Nisnevich sheaves
   $$\pi_{n+\ell,n}^{\bb A^1}(E)[p_1^{-1},\ldots,p_{t(\ell)}^{-1}]\cong H^{-\ell-n}(\bb Z[p_1^{-1},\ldots,p_{t(\ell)}^{-1}]^E(-n)).$$
The same statement is true with $1/p$-coefficients whenever $k$ is imperfect.
\end{thm}

\begin{proof}
We use induction in $\ell$. If $\ell=1$ then $\pi^s_1=\mathbb Z/2\mathbb Z$ and
$
\pi^{\mathbb A^1}_{n+1,n}(E)[2^{-1}] \cong H^{-1-n}\big(\mathbb Z[2^{-1}]^E(-n)\big)
$
by Corollary~\ref{sandin}. For any $q\leqslant \ell-1$ Theorem~\ref{ovechkin} implies isomorphisms of
presheaves
\[
\pi_{<q}(M_E(Y,\bb S^{\geq q}))=0
\qquad\text{and}\qquad
\pi_q(M_E(Y,\bb S^{\geq q}))=\pi_0(LM_E(Y))\otimes \pi^s_q.
\]
The tower~\eqref{shtower} now yields an exact sequence of Nisnevich sheaves
\begin{equation*}
\begin{aligned}
&\big(H^{-n}\bigl(\mathbb Z^E(-n)\bigr)\otimes \pi^s_\ell\big)
[p_1^{-1},\ldots,p_{t(\ell-1)}^{-1}] \to{}\\
&\qquad\to \pi^{\mathbb A^1}_{n+\ell,n}(E)[p_1^{-1},\ldots,p_{t(\ell-1)}^{-1}]
\to H^{-\ell-n}\big(\mathbb Z[p_1^{-1},\ldots,p_{t(\ell-1)}^{-1}]^E(-n)\big)\to 0.
\end{aligned}
\end{equation*}
This shows the right-hand map is surjective at the coarser localization but does not
yet identify $\pi^{\mathbb A^1}_{n+\ell,n}(E)$ at the finer one claimed in the
statement, since the leftmost term need not vanish there. We complete the argument
as follows.

Localization is an exact functor on Abelian groups (flat base change along
$\mathbb Z\to \mathbb Z[p_1^{-1},\ldots,p_{t(\ell)}^{-1}]$), so the displayed
sequence remains exact after further inverting the additional primes
$p_{t(\ell-1)+1},\ldots,p_{t(\ell)}$; note
$\{p_1,\ldots,p_{t(\ell)}\}\supseteq\{p_1,\ldots,p_{t(\ell-1)}\}$ by construction, the
former being obtained from the latter by adjoining the prime divisors of
$|\pi^s_\ell|$. This yields the exact sequence
\begin{equation*}
\begin{aligned}
&\big(H^{-n}\bigl(\mathbb Z^E(-n)\bigr)\otimes \pi^s_\ell\big)
[p_1^{-1},\ldots,p_{t(\ell)}^{-1}] \to{}\\
&\qquad\to \pi^{\mathbb A^1}_{n+\ell,n}(E)[p_1^{-1},\ldots,p_{t(\ell)}^{-1}]
\to H^{-\ell-n}\big(\mathbb Z[p_1^{-1},\ldots,p_{t(\ell)}^{-1}]^E(-n)\big)\to 0.
\end{aligned}
\end{equation*}
Now $\pi^s_\ell$ is finite, and every prime divisor of $|\pi^s_\ell|$ occurs among
$p_1,\ldots,p_{t(\ell)}$ by construction; hence $|\pi^s_\ell|$ is invertible in
$\mathbb Z[p_1^{-1},\ldots,p_{t(\ell)}^{-1}]$. Since $|\pi^s_\ell|$ annihilates
$\pi^s_\ell$, this forces
\[
\pi^s_\ell[p_1^{-1},\ldots,p_{t(\ell)}^{-1}]
= \pi^s_\ell\otimes_{\mathbb Z}\mathbb Z[p_1^{-1},\ldots,p_{t(\ell)}^{-1}] = 0.
\]
As localization commutes with $\otimes_{\mathbb Z}$, the leftmost term of the
displayed sequence vanishes:
\[
\big(H^{-n}\bigl(\mathbb Z^E(-n)\bigr)\otimes \pi^s_\ell\big)
[p_1^{-1},\ldots,p_{t(\ell)}^{-1}]
\cong H^{-n}\bigl(\mathbb Z^E(-n)\bigr)
\otimes\big(\pi^s_\ell[p_1^{-1},\ldots,p_{t(\ell)}^{-1}]\big) = 0.
\]
The sequence therefore degenerates to an isomorphism of Nisnevich sheaves
$
\pi^{\mathbb A^1}_{n+\ell,n}(E)[p_1^{-1},\ldots,p_{t(\ell)}^{-1}]
\xrightarrow{\cong}\ H^{-\ell-n}\big(\mathbb Z[p_1^{-1},\ldots,p_{t(\ell)}^{-1}]^E(-n)\big),
$
completing the induction. \qedhere
\end{proof}

We finish the section with the motivic Adams spectral sequence (MASS) for framed motives.

\begin{thm}[MASS for framed motives]\label{massthom}
Suppose $N$ is an $S^1$-spectrum and $E$ is a motivic Thom $T$-spectrum. 
Let $X,U\in\Sm_k$, $Z$ be closed and smooth in $X$ and $Y=X/(X-Z)$. Suppose that $N$ has an Adams resolution of the form~\eqref{adams}.

$(1)$ The Adams resolution~\eqref{adams} yields a tower in $SH$ of the form
   \begin{equation}\label{masstower} 
    \xymatrix{\cdots\ar[r]&M_E(Y,N^2)(U)\ar[d]\ar[r]&M_E(Y,N^1)(U)\ar[d]\ar[r]&M_E(Y,N)(U)\ar[d]\\
                                     &C_*(\bb ZF^E(U,Y)\otimes L^2)&C_*(\bb ZF^E(U,Y)\otimes L^1)&C_*(\bb ZF^E(U,Y)\otimes L^0)}
   \end{equation}

$(2)$ If $k$ is perfect,
the tower~\eqref{masstower} yields a tower of $\bb A^1$-local motivic $S^1$-spectra in $SH_{S^1}(k)$
   \begin{equation}\label{shmasstower} 
    \xymatrix{\cdots\ar[r]&M_E(Y,N^2)\ar[d]\ar[r]&M_E(Y,N^1)\ar[d]\ar[r]&M_E(Y,N)\ar[d]\\
                                     &C_*(\bb ZF^E(Y)\otimes L^2)_{\nis}&C_*(\bb ZF^E(Y)\otimes L^1)_{\nis}&C_*(\bb ZF^E(Y)\otimes L^0)_{\nis}}
   \end{equation}
The same statement is true with $1/p$-coefficients whenever $k$ is imperfect.

$(3)$ Let $k$ be perfect. Suppose that for each $i\geq 0$ there is $n\geq 0$ such that $N^q$ is $i$-connected for $q\geq n$.
Then the tower~\eqref{shmasstower} yields a strongly convergent spectral sequence
   $$E^2_{p,q}=\pi_{p+q}(C_*(\bb ZF^E(Y)\otimes L^q)^f)(U)\Longrightarrow\pi_{p+q}(M_E(Y,N)^f(U)),$$
where $C_*(\bb ZF^E(Y)\otimes L^q)^f$ and $M_E(Y,N)^f$ are stable local fibrant replacements of 
$C_*(\bb ZF^E(Y)\otimes L^q)$ and $M_E(Y,N)$. In particular, it induces  a strongly convergent spectral sequence
of Nisnevich sheaves
   $$E^2_{p,q}=\pi_{p+q}^{\nis}(C_*(\bb ZF^E(Y)\otimes L^q)_{\nis})\Longrightarrow\pi_{p+q}^{\nis}(M_E(Y,N)).$$
The same statement is true with $1/p$-coefficients whenever $k$ is imperfect.
\end{thm}

\begin{proof}
The proof literally repeats that of Theorem~\ref{ovechkin}.
\end{proof}

\begin{rem}\label{remmass}
The technique developed in this section suggests possible extensions
to other Adams-type resolutions. One example is the Adams resolution
discussed in Example~\ref{adamsresex}. Another important case to
investigate is the classical mod-$p$ Adams resolution of the sphere
spectrum $\bb S$. Its residual terms do not become arbitrarily
connective, so the convergence criterion established above does not
apply directly. Identifying the abutment and its relation to
completion would therefore require a separate argument.
More generally, it would be worthwhile to extend this approach to
$F$-Adams resolutions, where $F$ is a ring spectrum such as $MU$ or
$BP$, with the aim of constructing motivic analogues of
Adams--Novikov spectral sequences for $E$-framed motives.
The layers of these resolutions need not be spectra in simplicial
Abelian groups, as required in the present results, so such an
extension would also require an appropriate analysis of their
images under the framed motive construction.

These questions lie beyond the results proved in this section
and are left for future work.
\end{rem}

\section{Computing framed motivic cohomology}\label{nikishin}


In the previous section, we constructed the motivic Atiyah--Hirzebruch spectral sequence and
established its relationship with framed motivic cohomology. The effectiveness of this spectral
sequence depends on our ability to understand and compute the corresponding cohomology groups.
The purpose of this section is to develop additional tools for such computations by exploiting the
symmetric-group actions that are intrinsic to the theory of framed correspondences.
The results obtained here also
form the key technical ingredient in the reconstruction theorem proved in the following section.


\subsection{The tame $M$-module structure and permutation free correspondences}

\begin{dfn}
We follow~\cite[Definition~5.1]{GG23}.
Suppose $\sigma:\bb P^{\w 1}\to T$ is the canonical morphism of pointed Nisnevich sheaves given by $(\{0\},\bb A^1,t)\in\Fr_1(\pt,\pt)$.
Let $E$ be a symmetric motivic Thom $T$-spectrum in
$Sp^\Sigma(\cc M,T)$. Given $X,Y\in\Sm_k$ and an open subset $V\subset Y$, define the symmetric $S^1$-spectrum 
$\Fr^{E,\Sigma}_*(X,Y/V)$ as follows.
First, let
   $$\Fr^{E,\Sigma}_n(X,Y/V):=\Hom_{\cc M}(X_+\wedge \bb P^{\wedge n},Y/V\wedge E_n\wedge S^n).$$
This simplicial set is pointed at the zeroth map. 
By definition, $\Fr^{E,\Sigma}_0(X,Y/V):=\Hom_{\cc M}(X_+,Y/V\wedge E_0)$.
Each $\Fr^{E,\Sigma}_n(X,Y/V)$ is a $\Sigma_n$-simplicial set. 
The left action of $\Sigma_n$ on
$\Fr^{E,\Sigma}_n(X,Y/V)$ is given by conjugation: for $f:X_+\wedge \bb P^{\wedge n}\to Y/V\wedge E_n\wedge S^n$
and $\pi\in\Sigma_n$ the morphism $\pi\cdot f$ is defined as
   \begin{equation}\label{sigmaact}
    X_+\wedge\bb P^{\wedge n}\xrightarrow{X_+\wedge\pi^{-1}}X_+\wedge\bb P^{\wedge n}\xrightarrow{f}Y/V\wedge E_n\wedge S^n
       \xrightarrow{Y/V\wedge\pi\wedge\pi}Y/V\wedge E_n\wedge S^n.
   \end{equation}

Second, the morphism $\sigma$ induces natural $(\Sigma_n\times\Sigma_k)$-equivariant maps 
   $$\Fr^{E,\Sigma}_n(X,Y/V)\wedge S^k\to\Fr^{E,\Sigma}_{n+k}(X,Y/V),$$
so that 
   $$\Fr^{E,\Sigma}_*(X,Y/V):=(\Fr^{E,\Sigma}_0(X,Y/V),\Fr^{E,\Sigma}_1(X,Y/V),\Fr^{E,\Sigma}_2(X,Y/V),\ldots)$$
becomes a symmetric $S^1$-spectrum.
\end{dfn}

\begin{lem}\label{sup}
$\pi_k(\Fr^{E,\Sigma}_*(X,Y/V))=\bb ZF^E(X,Y/V)\otimes\pi_k^s$ for any $k\in\bb Z$.
\end{lem}

\begin{proof}
The proof of~\cite[Lemma~6.10]{GG23} shows that this spectrum is stably equivalent to 
   $$\Fr^{E}(X,Y/V\wedge\bb S):=(\Fr^{E}(X,Y/V),\Fr^{E}(X,Y/V\wedge S^1),\Fr^{E}(X,Y/V\w S^2),\ldots).$$
Our proof now follows from Theorem~\ref{carbery}.
\end{proof}

Likewise, for any Abelian group $A$ denote by 
   $$\bb ZF^{E,\Sigma}_*(X,Y/V)\otimes A:=(\bb ZF^{E}_0(X,Y/V)\otimes A,\bb ZF^{E}_1(X,Y/V\otimes S^1)\otimes A,\bb ZF^{E}_2(X,Y/V\otimes S^2)\otimes A,\ldots),$$
where each $\bb ZF^{E}_n(X,Y/V\otimes S^n)\otimes A=\bb ZF^{E}_n(X,Y/V)\otimes A\otimes\bb Z(S^n)$. It is a symmetric $S^1$-spectrum with
symmetric groups action given by~\eqref{sigmaact}. It is stably equivalent to
   $$\bb ZF^{E}(X,Y/V\otimes\bb S)\otimes A:=(\bb ZF^{E}(X,Y/V)\otimes A,\bb ZF^{E}(X,Y/V\otimes S^1)\otimes A,\bb ZF^{E}(X,Y/V\otimes S^2)\otimes A,\ldots).$$
Note that this spectrum is the Eilenberg--MacLane spectrum of $\bb ZF^{E}(X,Y/V)\otimes A$.

By~\cite{Sch08} the Abelian group $\bb ZF^E(X,Y/V)\otimes A=\pi_0(\Fr^{E,\Sigma}_*(X,Y/V)\otimes A)=\pi_0(\bb ZF^{E,\Sigma}_*(X,Y/V)\otimes A)$ 
enjoys the structure of a tame $M$-module, where $M$
is the monoid of injective self-maps of the set $\omega=\{1,2,\ldots\}$ of positive natural numbers.
We recall this structure for the convenience of the reader (see~\cite{Sch08} for details).

Consider a category $I$ consisting of objects $\mathbf n=\{1,\dots,n\}$
for every non-negative integer $n$, including $\mathbf{0}=\emptyset$.
Its morphisms in $I$ are all injective maps.
An {\em $I$-functor} is a covariant functor from the category $I$
to the category of Abelian groups.

Let $X\in Sp_{S^1}^\Sigma$ be a symmetric $S^1$-spectrum and let
$k\in\bb Z$. We assign an $I$-functor $\underline{\pi}_kX$ 
to $X$. On objects, this $I$-functor is given by
   $$ (\underline{\pi}_kX)(\mathbf n) \ = \ \pi_{k+n}X_n $$
if $k+n\geq 2$ and $(\underline{\pi}_kX)(\mathbf n)=0$ for $k+n<2$.
If $\alpha:\mathbf n\to \mathbf m$ is an injective map and $k+n\geq 2$,
then $\alpha_*: (\underline{\pi}_kX)(\mathbf n)\to (\underline{\pi}_kX)(\mathbf m)$
is given as follows. We choose a permutation $\gamma\in\Sigma_m$
such that $\gamma(i)=\alpha(i)$ for all $i=1,\dots,n$ and set
   \begin{equation}\label{mustamae}
    \alpha_*(x)\ =sign(\gamma) \cdot \gamma_*(\iota_*^{m-n}(x))
   \end{equation}
where $\iota_*:\pi_{k+n}X_n\to\pi_{k+n+1}X_{n+1}$ is the 
stabilization map
   $$\pi_{k+n} \, X_n \ \xrightarrow{\ -\w S^1\ } \
       \pi_{k+n+1} \, \left( X_n\w S^1 \right) \ \xrightarrow{\ (\sigma_n)_*\ } \pi_{k+n+1} \, X_{n+1}.$$
The definition of $\alpha_*(x)$ is independent of the choice of $\gamma$.

\begin{lem}\label{zanyatno}
We have $(\underline{\pi}_0\,\bb ZF^{E,\Sigma}_*(X,Y/V)\otimes A)(\mathbf n)=\bb ZF^{E}_n(X,Y/V)\otimes A$ if $n\geq 2$ and
$\alpha_*(x)=\gamma_*(u_*^{m-n}(x))$ for any $x\in\bb ZF^E_n(X,Y/V)\otimes A$, where $u_*$
 is determined by the composite map
    $${\cc M}(X_+\wedge \bb P^{\wedge n},Y/V\wedge E_n)\xrightarrow{-\wedge\sigma}
        {\cc M}(X_+\wedge \bb P^{\wedge n+1},Y/V\wedge E_n\wedge T)\xrightarrow{u_n}
        {\cc M}(X_+\wedge \bb P^{\wedge n+1},Y/V\wedge E_{n+1})$$
and $\Sigma_m$ acts on $\bb ZF^{E}_m(X,Y/V)\otimes A$ by conjugation as in~\eqref{sigmaact}.
\end{lem}

\begin{proof}
By definition, 
   \begin{multline*}
    (\underline{\pi}_0\,\bb ZF^{E,\Sigma}_*(X,Y/V)\otimes A)(\mathbf n)=
    {\pi}_n(\,\bb ZF^{E}_n(X,Y/V)\otimes A\otimes\bb Z(S^n))=\\
       =\bb ZF^{E}_n(X,Y/V)\otimes A\otimes\pi_n(\bb Z(S^n))=\bb ZF^{E}_n(X,Y/V)\otimes A.
   \end{multline*}
 The action of $\gamma\in\Sigma_m$ on the simplicial Abelian group 
 $\bb ZF^{E}_m(X,Y/V)\otimes A\otimes\bb Z(S^m)$
equals $\gamma_*\otimes\gamma_*=(\gamma_*\otimes\id)\circ(\id\otimes\gamma_*)$, where
$\gamma_*\otimes\id$ (respectively $\id\otimes\gamma_*$) acts on the $\bb ZF^{E}_m(X,Y/V)\otimes A$-factor
(respectively on the $\bb Z(S^m)$-factor). But
   $\pi_m(\id\otimes\gamma_*)=sign(\gamma)\cdot(\id\otimes\id)$
due to commutativity of the diagram 
   $$\xymatrix{S^m\ar[r]^{t}\ar[d]_\gamma&\bb Z[S^m]\ar[d]^\gamma\\
                       S^m\ar[r]^{t}&\bb Z[S^m]}$$
for any $t\in\pi_m(\bb Z[S^m])$ and the effect on homotopy groups of precomposing with a 
coordinate permutation of the sphere is
multiplication by the sign of the permutation.

Now the general formula~\eqref{mustamae} takes the form
   $$\alpha_*(x)=(sign(\gamma))^2(\gamma_*\otimes\id)(u_*^{m-n}(x))=(\gamma_*\otimes\id)(u_*^{m-n}(x)).$$
Our lemma now follows.
\end{proof}

For any $I$-functor $F$ the colimit
of $F$, formed over the subcategory $\bb N$ of inclusions, 
has a natural left action by the monoid $M$ of injective endomorphisms
of the set $\omega$ of natural numbers. Applied to the $I$-functor
$\underline{\pi}_kX$ coming from a symmetric spectrum $X$, this yields
the $M$-action on the stable homotopy group $\pi_kX$.
In more detail (see~\cite[p.~1318]{Sch08}), set $F(\omega)=\colim_\bb N F$. First define
$\beta_*:F(\mathbf n)\to F(\omega)$ for every injection 
$\beta:\mathbf n\to\omega$
as follows. Let $m=\max\{\beta(\mathbf n)\}$ and let
$\beta|_{\mathbf n}:\mathbf n\to \mathbf m$ be
the restriction of $\beta$. Take $\beta_*(x)$ 
to be the class in the colimit 
represented by the image of $x$  under 
$$ (\beta|_{\mathbf n})_*\ :\  F(\mathbf n)\ \to \ F(\mathbf m). $$
This is a functorial extension of
$F$, i.e., for every map $\alpha:\mathbf k\to\mathbf n$ in $I$ one has
$(\beta\alpha)_*(x)=\beta_*(\alpha_*(x))$.

Next, let $f:\omega\to\omega$ be an injective endomorphism of $\omega$.
We want to define $f_*:F(\omega)\to F(\omega)$.
If $[x]\in F(\omega)$ is an element in the colimit represented 
by $x\in F(\mathbf n)$, one sets
$f_*[x]=[(f|_{\mathbf n})_*(x)]$, where $f|_{\mathbf n}:\mathbf n\to \omega$
is the restriction of $f$ and $f_*:F(\mathbf n)\to F(\omega)$
was defined above. This does not depend on the representative $x$
of the class $[x]$ in the colimit. Moreover, the extension is functorial in the sense that
$(f\alpha)_*(x)=f_*(\alpha_*(x))$ for injections 
$\alpha:\mathbf n\to \omega$ as well as
$(fg)_*[x]=f_*(g_*[x])$ when $g$ is another injective endomorphism of $\omega$.
Note that if $\iota:\mathbf n\to\omega$ is the inclusion,
then $\iota_*(x)=[x]$ for $x\in F(\mathbf n)$.

It is worth mentioning that the action of the monoid $M$ on the colimit of
any $I$-functor $F$ has a special property: 
every element in the colimit $F(\omega)$ is represented by a class
$x\in F(\mathbf n)$ for some $n\geq 0$; 
then for every element $f\in M$ which fixes 
the numbers $1,\dots,n$, we have $f_*[x]=[x]$.
If $F$ is an $I$-functor of Abelian groups, then such $M$-modules are referred to as {\it tame\/} in~\cite{Sch08}.

\begin{cor}\label{M-module}
$\bb ZF^E(Y/V,A)=\bb ZF^E(-,Y/V)\otimes A$ enjoys the structure of a presheaf of tame $M$-modules.
\end{cor}

Recall that stable equivalences of symmetric spectra are not detected by stable homotopy groups defined naively.
Instead, one defines ``true" stable homotopy groups for a symmetric spectrum $X\in Sp^\Sigma_{S^1}$, denoted
by $\wt\pi_*(X)$. There is a strongly convergent half-plane spectral sequence 
   \begin{equation} \label{schwede} 
    E^2_{p,q} = \Tor_p^{\bb Z[M]}(\bb Z,\pi_q X) \ \Longrightarrow \wt\pi_{p+q}(X). 
    \end{equation}
relating ordinary stable homotopy groups to true ones. The spectral sequence is natural in~$X$ with 
$d^r$-differential of bidegree $(-r,r-1)$ --- see~\cite[Section~5]{Sch08} for details.

This spectral sequence together with Lemma~\ref{sup} and Corollary~\ref{M-module} implies the following statement.

\begin{lem}\label{naivevstrue}
There is a strongly convergent half-plane spectral sequence 
   \begin{equation*} \label{schwede1} 
    E^2_{p,q} = \Tor_p^{\bb Z[M]}(\bb Z,\bb ZF^E(X,Y/V)\otimes\pi_q^s) \ \Longrightarrow \wt\pi_{p+q}(\Fr^{E,\Sigma}_*(X,Y/V)). 
    \end{equation*}
\end{lem}

\begin{cor}\label{naivevstruecor}
There is an isomorphism of Abelian groups 
   $$\bb Z\otimes_M\bb ZF^E(X,Y/V)\cong\wt\pi_0(\Fr^{E,\Sigma}_*(X,Y/V)),$$ 
where $\bb Z$ is regarded as an $M$-module with trivial $M$-action.
\end{cor}

\begin{proof}
This follows from Lemma~\ref{naivevstrue}.
\end{proof}

The next result introduces an additive category which will play an important role in our analysis.

\begin{thm}\label{catcor}
Let $E$ be a symmetric motivic Thom ring $T$-spectrum. There is an
additive category, denoted by $\wt{\bb Z}F^E(k)$, whose objects are
those of $\Sm_k$ and whose morphism groups are
$\bb Z\otimes_M\bb ZF^E(X,Y)$.
Finite coproducts are given by disjoint unions of schemes.
If $E$ is commutative, then $\wt{\bb Z}F^E(k)$ is symmetric monoidal,
with monoidal product given by the product of schemes and unit
$\operatorname{Spec}k$.
The unit morphism $S_T\to E$ induces a canonical additive functor
   $$\Phi:\bb ZF_*(k)\longrightarrow\wt{\bb Z}F^E(k),$$
which is the identity on objects and satisfies
$\Phi(\sigma_X)=\mathrm{id}_X$ for every $X\in\Sm_k$.
\end{thm}

\begin{proof}
By~\cite[Theorem~5.2]{GG23} $\Sm_k$ is enriched over $Sp_{S^1}^\Sigma$. Its symmetric spectra of morphisms are given by
$\Fr_*^{E,\Sigma}(X,Y)$. This spectral category  is symmetric monoidal if $E$ is a commutative ring $T$-spectrum. The construction of the
additive functor $\Phi$ is straightforward.
Our statement now follows from Corollary~\ref{naivevstruecor} and~\cite[Theorem~I.6.16]{Sch}.
\end{proof}

We want to give another description of the Abelian groups $\wt{\bb Z}F^E(X,Y/V)$, where $E$ is a symmetric motivic Thom
$T$-spectrum. To this end, we will define the structure of a tame $M$-set on the framed sheaf $\Fr^E(Y/V)$.
Consider an $I$-functor $\underline{\Fr}^E(Y/V)$ from $I$ to presheaves of pointed sets. On
objects, this $I$-functor is given by
   $$\underline{\Fr}^E(Y/V)(\mathbf n):=\Fr_n^E(Y/V).$$
If $\alpha:\mathbf n\to \mathbf m$ is an injective map,
then $\alpha_*:\underline{\Fr}^E(Y/V)(\mathbf n)\to\underline{\Fr}^E(Y/V)(\mathbf m)$
is given as follows. We choose a permutation $\gamma\in\Sigma_m$
such that $\gamma(i)=\alpha(i)$ for all $i=1,\dots,n$ and set
   \begin{equation*}
    \alpha_*(x):=\gamma_*(\sigma_*^{m-n}(x))
   \end{equation*}
where $\iota_*:\Fr_n(Y/V)\to\Fr_{n+1}(Y/V)$ is the 
stabilization map 
   \begin{equation}\label{taktak}
    \uhom(\PP^{\w n},Y/V\w E_n)\xrightarrow{-\wedge\sigma}\uhom(\PP^{\w n}\wedge\PP^{\w 1},Y/V\w E_n\wedge T)
       \xrightarrow{u_n}\uhom(\PP^{\w n+1},Y/V\w E_{n+1}).
   \end{equation}

We claim that the definition of $\alpha_*(x)$ is independent of the choice of $\gamma$.
Suppose $\gamma'\in\Sigma_m$ is another permutation which agrees with $\alpha$
on $\mathbf n$. Then $\gamma^{-1}\gamma'$ is a permutation
of $\mathbf m$ which fixes the numbers $1,\dots,n$, so it is of the form
$\gamma^{-1}\gamma'=1\times\tau$ for some $\tau\in\Sigma_{m-n}$, 
where $1$ is the unit of $\Sigma_n$. As $E$ is a symmetric $T$-spectrum, one has a commutative diagram
   $$\xymatrix@C=15mm{ \bb P^{\w n+(m-n)} \ar[r]^-{x\w \sigma^{m-n}} \ar[d]_{\id\w\tau} &
       Y/V\w E_n\w T^{m-n} \ar[r]^-{u_*^{m-n}} \ar[d]^{\id\w\tau} &
       Y/V\w E_m \ar[d]^{1\times\tau}\\
       \bb P^{\w n+(m-n)} \ar[r]_-{x\w\sigma^{m-n}} &
       Y/V\w E_n\w T^{m-n} \ar[r]_{u_*^{m-n}} & Y/V\w E_m }$$
It follows that
   $(1\times\tau)_*(\iota^{m-n}_*(x))=\iota^{m-n}_*(x),$
and hence the claim.

The inclusion $\mathbf n\hookrightarrow \mathbf{n+1}$ induces the map $\iota_*$ 
over which the colimit $\Fr^E(Y/V)$ is formed. 
Denote by $\bb N$ the subcategory of $I$ containing all objects 
but only the inclusions as morphisms. Then
   $$\Fr^E(Y/V)= \colim_{\bb N}\ \underline{\Fr}^E(Y/V).$$

By construction, the value of the extension
at the object $\omega$ is the colimit of $F$ formed over the subcategory $\bb N$ of inclusions, denoted by $F(\omega)$.
Then the $M$-action on the colimit of $F$ is the action of the endomorphisms of $\omega$ in $I_\omega$ on $F(\omega)$.
In this way we get the pointed $M$-set $\Fr^E(X,Y/V)$. Denote by 
   $$\wt\Fr^E(X,Y/V)=\Fr^E(X,Y/V)/\{mx\sim x\mid m\in M\},$$
   $$\wt{\bb Z} F^E_n(X,Y/V)=\bb ZF_n^E(X,Y/V)/\{\tau x\sim x\mid\tau\in\Sigma_n\}=\bb Z\otimes_{\Sigma_n}\bb ZF_n^E(X,Y/V),$$
and
   $$\wt\Fr^E_n(X,Y/V)=\Fr_n^E(X,Y/V)/\{\tau x\sim x\mid\tau\in\Sigma_n\}.$$
The stabilization map $\iota_*:\Fr_n^E(Y/V)\to\Fr_{n+1}^E(Y/V)$ defined in~\eqref{taktak} induces the
stabilization map $i_*:\wt\Fr_n^E(X,Y/V)\to\wt\Fr^E_{n+1}(X,Y/V)$.
Observe that $i_*$ is functorial in $X$, and hence it induces a map of presheaves
$i_*:\wt\Fr_n^E(Y/V)\to\wt\Fr_{n+1}^E(Y/V)$.

\begin{lem}\label{rakvere}
The natural map $\alpha:\colim_{n}\wt\Fr_n^E(Y/V)\to\wt\Fr^E(Y/V)$
is bijective. As a consequence, the natural morphism 
   $$\colim_n\bb Z\otimes_{\Sigma_n}\bb ZF_n^E(Y/V)=\colim_{n}\wt{\bb Z} F_n^E(Y/V)\to\wt{\bb Z} F^E(Y/V)$$ 
of presheaves of free Abelian groups is an isomorphism.
\end{lem}

\begin{proof}
Fix $X\in\Sm_k$. For any functor $D:I\to\mathrm{Set}_*$,
there are natural bijections
\[
\colim_n D(\mathbf n)/\Sigma_n
\cong \colim_I D
\cong
\bigl(\colim_{\mathbb N}D\bigr)/M.
\]
Indeed, a compatible family of pointed maps from
$D(\mathbf n)/\Sigma_n$ to a pointed set $S$ is 
a family of maps $D(\mathbf n)\to S$ compatible with the
standard inclusions and invariant under the symmetric
groups. Since every injection $\mathbf n\to\mathbf m$ is a standard
inclusion followed by a permutation of $\mathbf m$, such
families are precisely those compatible with every morphism
of $I$, and hence correspond uniquely to pointed maps
$\colim_I D\to S$.
Likewise, a pointed map $\colim_{\mathbb N}D\to S$
is $M$-invariant if and only if the corresponding family
is compatible with every morphism of $I$. Here one uses
the fact that every injection between finite subsets of
$\omega$ extends to an element of $M$.
Taking $D=\underline{\Fr}^E(X,Y/V)$ proves the assertion
about $\alpha$.

For every $I$-functor $A:I\to\mathrm{Ab}$, 
the same argument in Abelian groups gives natural isomorphisms
\[
\colim_n A(\mathbf n)_{\Sigma_n}
\cong \colim_I A
\cong
\mathbb Z\otimes_M\colim_{\mathbb N}A.
\]
Applying this to $A(\mathbf n)=\mathbb ZF_n^E(X,Y/V)$
proves the asserted linear isomorphism, natural in $X$.

We finally verify freeness. For a correspondence
$c\in\Fr_n^E(X,Y/V)$, the supports of its successive
stabilizations satisfy
\[
\operatorname{Supp}(\operatorname{st}^{r+1}c)
\subseteq
\operatorname{Supp}(\operatorname{st}^{r}c)\times\{0\}.
\]
After identifying these supports with closed subsets
of $\operatorname{Supp}(c)$, they form a descending chain
in a Noetherian space and therefore stabilize.
At a sufficiently large stage, decomposition into
connected components consequently expresses the image
of $[c]$ as a sum of generators whose every subsequent
stabilization has nonempty connected support.

Let $C_n$ be the set of nonzero connected support
correspondences at level $n$ with this latter property.
Stabilization and permutations preserve these sets,
so they form an $I$-functor. Put
$B_n=\mathbb Z[C_n]\subseteq\mathbb ZF_n^E(X,Y/V)$.
Every element of $\mathbb ZF_n^E(X,Y/V)$ maps into some
$B_m$. Since filtered colimits preserve monomorphisms,
the canonical map
\[
\colim_n B_n\longrightarrow
\colim_n\mathbb ZF_n^E(X,Y/V)
\]
is an isomorphism. Setting
$C_\infty=\colim_{\mathbb N}C_n$, we obtain an
$M$-equivariant isomorphism
$\mathbb ZF^E(X,Y/V)\cong\mathbb Z[C_\infty]$.
We see that $\widetilde{\mathbb Z}F^E(X,Y/V)\cong\mathbb Z[C_\infty/\!\sim_M]$,
where $\sim_M$ is generated by $mc\sim c$.
This group is free Abelian, and the linear isomorphism
already proved establishes freeness of its source.
\end{proof}

The previous lemma justifies the following terminology.

\begin{dfn}\label{permfree}
Let $E$ be a symmetric motivic Thom ring $T$-spectrum. The additive category $\wt{\bb Z}F^E(k)$ of Theorem~\ref{catcor}
will be referred to as the {\it category of permutation free linear $E$-framed correspondences}. 
``Permutation free" here means that permutation actions have been quotiented out. In the special case when
$E=S_T$ we will omit $E$ from notation
and refer to the category $\wt{\bb Z}F(k)$ as the {\it category of permutation free linear framed correspondences}.
\end{dfn}

\begin{rem}\label{permfreerem}
Let $E$ be a symmetric motivic Thom ring $T$-spectrum. We also have a category, denoted by $\wt\Fr^E(k)$, whose objects are those of
$\Sm_k$ and morphisms are given by the sets $\wt\Fr^E(X,Y)$, $X,Y\in\Sm_k$.
Lemma~\ref{rakvere} describes the sets as $\colim_n\Fr_n^E(X,Y)/\Sigma_n$. For this reason we refer to $\wt\Fr^E(k)$ as the
{\it category of permutation free $E$-framed correspondences}. Observe that there is a natural functor between categories
$\Fr_*(k)\to\wt\Fr^E(k)$.
\end{rem}

Let $E$ be a symmetric motivic Thom $T$-spectrum and let $A$ be an Abelian group. Choose a free resolution $F_*\to\bb Z$
in the category of $M$-modules such that $F_0=\bb Z[M]$. We have
   $$\Tor_p^{\bb Z[M]}(\bb Z,\bb ZF^E(X,Y/V)\otimes A)=H_p(F_*\otimes_M\bb ZF^E(X,Y/V)\otimes A).$$
We have that $F_*\otimes_M\bb ZF^E(Y/V,A)$ is a chain 
complex of presheaves of Abelian groups. Therefore $\Tor_p^{\bb Z[M]}(\bb Z,\bb ZF^E(Y/V,A))$
is a presheaf. We denote this presheaf by $\bb ZF_{\tor}^E(p,Y/V,A)$.

\begin{prop}\label{R-module}
Let $R$ be a symmetric motivic Thom ring $T$-spectrum. Suppose the motivic Thom $T$-spectrum $E$
is a right $R$-module. Then $\bb ZF_{\tor}^E(p,Y/V,A)$ is a $\wt{\bb Z}F^R(k)$-presheaf for all $p\geq 0$. In particular, it is a 
stable $\bb ZF_*(k)$-presheaf.
\end{prop}

\begin{proof}
By~\cite[Theorem~5.2]{GG23} $\Sm_k$ is enriched over symmetric $S^1$-spectra whose morphism spectra
are given by $\Fr_*^{R,\Sigma}(X,Y)$, $X,Y\in\Sm_k$. Denote this spectral category by $\Fr_*^{R,\Sigma}(k)$.
As $E$ is a right $R$-module, $\Fr_*^{E,\Sigma}(-,Y/V)$ enjoys the structure of a right $\Fr_*^{R,\Sigma}(k)$-module,
where the relevant pairings literally repeat pairings from~\cite[Theorem~5.2]{GG23}. In a similar fashion, one has 
a linear spectral category $\bb Z F_*^{R,\Sigma}(k)$ whose symmetric spectra of morphisms are given by
$\bb Z F_*^{R,\Sigma}(X,Y)$, $X,Y\in\Sm_k$, as well as a right $\bb Z F_*^{R,\Sigma}(k)$-module $\bb Z F_*^{E,\Sigma}(-,Y/V)$.
Tensoring this module with $A$ produces a right $\bb Z F_*^{R,\Sigma}(k)$-module $\bb Z F_*^{E,\Sigma}(-,Y/V)\otimes A$.

As we have noted in the beginning of the section, $\bb Z F_*^{E,\Sigma}(-,Y/V)\otimes A$ is stably equivalent to
the Eilenberg–MacLane spectrum of 
$\bb Z F^{E}(-,Y/V)\otimes A$, and hence it has only one presheaf of stable homotopy groups 
$\pi_0(\bb Z F_*^{E,\Sigma}(-,Y/V)\otimes A)=\bb Z F^{E}(-,Y/V)\otimes A$. So the spectral
sequence~\eqref{schwede} for $\bb Z F_*^{E,\Sigma}(-,Y/V)\otimes A$ collapses onto the axis $q= 0$ to isomorphisms
   $$\tilde\pi_p(\bb Z F_*^{E,\Sigma}(-,Y/V)\otimes A)\cong\bb ZF_{\tor}^E(p,Y/V,A).$$
In particular, the true homotopy groups of $\bb Z F_*^{E,\Sigma}(-,Y/V)\otimes A$ need not be concentrated in degree 0.
The fact that $\bb ZF_{\tor}^E(p,Y/V,A)$ is a $\wt{\bb Z}F^R(k)$-presheaf for all $p\geq 0$ follows from
Corollary~\ref{naivevstruecor} and~\cite[Theorem~I.6.16]{Sch}.

By Theorem~\ref{catcor} there is a natural functor
   $$\Phi:\bb ZF_*(k)\to\wt{\bb Z}F^R(k).$$
Note that $\Phi(\sigma_X:X_+\w\bb P^{\w 1}\to X_+\w T)=\id_X$. So
$\bb ZF_{\tor}^E(p,Y/V,A)$ is a stable $\bb ZF_*(k)$-presheaf.
\end{proof}

\subsection{Torsion framed motivic cohomology}
Let $\Gr\Ab$ denote the closed symmetric monoidal category of graded Abelian 
groups (under graded tensor product, with Koszul sign convention for the symmetry isomorphism).
In fact, the category of permutation free framed correspondences are degree zero correspondences of 
a graded category of correspondences. More precisely, the proof of the preceding proposition together with
Theorem~\ref{catcor} imply the following statement.

\begin{thm}\label{torsioncorrs}
Let $R$ be a symmetric motivic Thom ring $T$-spectrum. The category $\Sm_k$ is enriched over $\Gr\Ab$
with morphism objects given by $\bigoplus_{p\geq 0}\bb ZF^R_{\tor}(p,X,Y)\in\Gr\Ab$,
$X,Y\in\Sm_k$. We call it the \emph{category of linear $(R,\Tor)$-framed correspondences\/} and denote it by $\wt{\bb Z}F^R_{\tor}(k)$. 
The elements of 
$\bb ZF^R_{\tor}(p,X,Y)$ are called \emph{linear $(R,\Tor)$-framed correspondences of level $p$}.
This category is symmetric monoidal whenever $R$ is commutative. Moreover,
if a motivic Thom $T$-spectrum $E$ is a right $R$-module then $\bigoplus_{p\geq 0}\bb ZF_{\tor}^E(p,Y/V,A)$ is a right
$\wt{\bb Z}F^R_{\tor}(k)$-module for any Abelian group $A$.
\end{thm}

\begin{dfn}\label{tormotcohom}
Let $E$ be a symmetric motivic Thom $T$-spectrum. 
For every integers $p,q\geq 0$ the motivic complex $\bb Z^E_{\tor}(p,q)$ is defined
as the following (bounded above) cochain complex 
   $$\bb Z^E_{\tor}(p,q)=C_*\bb ZF_{\tor}^E(p,\bb G^{\w q})_{\nis}[-q]$$
of stable $\bb ZF_*(k)$-sheaves (the shift is cohomological).
If $A$ is any other Abelian group then $A^E_{\tor}(p,q) $ is the cochain complex of
stable $\bb ZF_*(k)$-sheaves $C_*\bb ZF^E_{\tor}(p,\bb G^{\w q},A)_{\nis}[-q]$. If $q<0$ we set 
   $$A^E_{\tor}(p,q):=\uhom(\bb G^{\w(-q)},A_{\tor}^E(p,0))[-q].$$
   
The {\it $(E,\Tor_p)$-framed motivic cohomology groups $H^{n,q}_{E}(X,A,\Tor_p)$ with coefficients in $A$} are defined to
be the hypercohomology of the motivic complexes $A^E_{\tor}(p,q)$ with respect to the Nisnevich
topology:
   $$H^{n,q}_{E}(X,A,\Tor_p):=H_{\nis}^n(X,A^E_{\tor}(p,q)).$$
If $p\geq 1$ we will also refer to the groups $H^{n,q}_{E}(X,A,\Tor_p)$ as {\it torsion $E$-framed motivic cohomology groups}.
If $E=S_T$ we will write $A^{fr}_{\tor}(p,q)$ for $A_{\tor}^E(p,q)$ and $H^{n,q}_{fr}(X,A,\Tor_p):=H_{\nis}^n(X,A^{fr}_{\tor}(p,q)).$

For $q<0$ we set $\mathbb ZF_{\mathrm{tor}}^E(p,\gmp^{\wedge q},A):=
\underline{\mathrm{Hom}}\bigl(\mathbb G^{\wedge(-q)},\mathbb ZF_{\mathrm{tor}}^E(p,0,A)\bigr)$.
\end{dfn}

The terminology of the preceding definition is justified by the following useful statement.

\begin{lem}\label{torsionfr}
$\bb ZF_{\tor}^E(p,\gmp^{\w q},A)\otimes\bb Q=0$ for any $p\geq 1$ and $q\in\bb Z$. 
\end{lem}

\begin{proof}
This follows from the fact that $\Tor^{\bb Z[M]}_p(\bb Q,W)=0$ for every tame $M$-module $W$ and all $p\geq 1$ ---
see~\cite[Example~5.7]{Sch08}. It remains to apply Corollary~\ref{M-module}.
\end{proof}

We are now in a position to prove the main result of the section.

\begin{thm}\label{rakushka}
Suppose $k$ is a perfect field. Let $E$ be a symmetric motivic Thom 
$T$-spectrum with contractible alternating group action and let $A$ be an Abelian group.
There is a strongly convergent spectral sequence
   $$E^2_{s,t}=H_{E}^{-s,q}(U,A,\Tor_t)\Longrightarrow H_{E}^{-s-t,q}(U,A),\quad q\in\bb Z,$$
relating $E$-framed motivic cohomology to $(E,\Tor_t)$-framed motivic cohomology. If $k$ is imperfect,
the same statement is true with $1/p$-coefficients.
\end{thm}

\begin{proof}
By Corollary~\ref{M-module} $C_*\bb ZF^E(\bb G^{\w q},A)$ is a chain complex of presheaves of $\bb Z[M]$-modules.
Denote by $\cc K_0=F_*\otimes_M\bb ZF^E(\bb G^{\w q},A)$ and
$\cc L_0=\Tot(F_*\otimes_MC_*\bb ZF^E(\bb G^{\w q},A))$, where $F_*\to\bb Z$ is a resolution of $\bb Z$
by free $\bb Z[M]$-modules with $F_0=\bb Z[M]$. If we regard $\bb ZF^E(\bb G^{\w q},A)$ as a $\bb Z[M]$-module
concentrated in zeroth degree, the canonical map of chain complexes of $\bb Z[M]$-modules
$\kappa:\bb ZF^E(\bb G^{\w q},A)\to\cc K_0$ induces a map of chain complexes
   $$\lambda:C_*\bb ZF^E(\bb G^{\w q},A)\to\cc L_0.$$
   
We claim that $\lambda$ is a sectionwise quasi-isomorphism of chain complexes of Abelian groups.
Due to our assumption that $E$ is a symmetric motivic Thom $T$-spectrum with contractible alternating group action
as well as the fact that $\bb A^1$-homotopies become the usual ones after applying the Suslin complex $C_*$,
it follows from~\cite[Corollary~2.13]{GN} that even permutations from $\Sigma_n$ act trivially on the homology presheaves of
$C_*\bb ZF^E_n(\bb G^{\w q},A)$. Therefore the monoid $M$ acts trivially on the homology presheaves of
$C_*\bb ZF^E(\bb G^{\w q},A)$ due to~\cite[Example~3.2]{Sch08}. As $F_a$ is a free $\bb Z[M]$-module for every $a\geq 0$,
the same is true for the homology presheaves of the chain complex $F_a\otimes_M C_*\bb ZF^E(\bb G^{\w q},A)$
as $H_*(F_a\otimes_M C_*\bb ZF^E(\bb G^{\w q},A))=F_a\otimes_MH_*(C_*\bb ZF^E(\bb G^{\w q},A))$.
The groups $\Tor^{\bb Z[M]}_{s>0}(\bb Z,H_*(C_*\bb ZF^E(\bb G^{\w q},A)))$ 
vanish by \cite[Lemma~5.3(ii)]{Sch08}. The convergent spectral sequence for double complexes~\cite[Lemma ~12.25.3]{Stack} 
takes the form
   $$E^2_{s,t}=H_s(H_t(F_*\otimes_MC_*\bb ZF^E(\bb G^{\w q},A)))=
       \Tor^{\bb Z[M]}_s(\bb Z,H_t(C_*\bb ZF^E(\bb G^{\w q},A)))\Longrightarrow H_{s+t}(\cc L_0).$$
The spectral sequence now implies $\lambda$ is a quasi-isomorphism, as claimed.

There is a tower of chain complexes
   \begin{equation*}\label{ttower}
    \cdots\to\cc L_2\to\cc L_1\to\cc L_0,
   \end{equation*}
where each $\cc L_i=\Tot(C_*(\tau_{\geq i}\cc K_0))$ with $\tau_{\geq i}\cc K_0$ being the $i$-th
truncation of $\cc K_0$ --- see~\cite[Section~12.15]{Stack}. By construction, $\cc L_i$ is $(i-1)$-connected.
We have $\coker(\tau_{\geq i+1}\cc K_0\to\tau_{\geq i}\cc K_0)$ equals the complex
$\im d_{i+1}\hookrightarrow\kr d_i$ placed in the $(i+1)$-th and $i$-th degrees. The latter complex is quasi-isomorphic to
$H_i(\cc K_0)=\Tor^{\bb Z[M]}_i(\bb Z,\bb ZF^E(\bb G^{\w q},A))$ placed in the $i$-th degree. Thus
$C_*(H_i(\cc K_0))=C_*\bb Z^E_{\tor}(i,\bb G^{\w q},A)[-i]$ (the shift is homological) is sectionwise quasi-isomorphic to
the $i$-th layer of the tower above. After sheafifying this tower in the Nisnevich topology, one gets a tower
      \begin{equation}\label{dtower}
       \cdots\to\cc L_{2,\nis}\to\cc L_{1,\nis}\to\cc L_{0,\nis},
   \end{equation}      
where each $\cc L_{i,\nis}$ is $(i-1)$-connected as a chain complex of Nisnevich sheaves. By construction,
the $i$-th layer is quasi-isomorphic to the chain complex $C_*\bb Z^E_{\tor}(i,\bb G^{\w q},A)_{\nis}[-i]$ (the shift is homological).
By Proposition~\ref{R-module} the latter is a chain complex of stable $\bb ZF_*(k)$-sheaves.
If $k$ is perfect (respectively imperfect) this complex has strictly $\bb A^1$-invariant cohomology sheaves
(with $1/p$-coefficients), and hence it is 
$\bb A^1$-local by~\cite[Theorem~6.2.7]{Mor}. It follows that the local fibrant replacement 
$C_*\bb Z^E_{\tor}(i,\bb G^{\w q},A)_{\nis}^f$ is motivically fibrant (respectively motivically fibrant  with $1/p$-coefficients).

The tower~\eqref{dtower} induces a tower of motivically fibrant complexes
      \begin{equation}\label{mtower}
       \cdots\to\cc L_{2}^{mf}\to\cc L_{1}^{mf}\to\cc L_{0}^{mf}.
   \end{equation}  
As $\cc L_{i,\nis}$ is locally $(i-1)$-connected, so is $\cc L_i^{mf}$ by Morel's
Connectivity Theorem~\cite{Mor}.
If $U\in\Sm_k$ is of dimension $d$, then homology groups of the chain complex
$\cc L_i^{mf}(U)$ vanish in degrees smaller than $i-d$
by~\cite[Lemma~4.3.1]{Mor}. By~\cite[Corollary~6.1.1]{FS} the tower~\eqref{mtower}
produces a strongly convergent spectral sequence
   $$E^2_{s,t}=H_{s+t}(C_*\bb Z^E_{\tor}(t,\bb G^{\w q},A)_{\nis}^f(U)[-t])\Longrightarrow H_{s+t}(\cc L_{0}^{mf}(U)).$$
As we have shown above $\cc L_0$ is sectionwise quasi-isomorphic to $C_*\bb ZF^E(\bb G^{\w q},A)$. So
$\cc L_{0,\nis}$ is locally quasi-isomorphic to $C_*\bb ZF^E(\bb G^{\w q},A)_{\nis}$ and we may assume that $\cc L_{0}^{mf}=\cc L_{0,\nis}^f$.
We see that 
   $$H_{s+t+q}(\cc L_{0,\nis}^f(U))\cong H_{\nis}^{-s-t}(U,A^E(q))=H_{E}^{-s-t,q}(U,A)$$
and
   \begin{multline*}
    H_{s+t+q}(C_*\bb Z^E_{\tor}(t,\bb G^{\w q},A)_{\nis}^f(U)[-t])=H_{s+q}(C_*\bb Z^E_{\tor}(t,\bb G^{\w q},A)_{\nis}^f(U))\cong\\
    \cong H_{\nis}^{-s}(U,A^E_{\tor}(t,q))=H_{E}^{-s,q}(U,A,\Tor_t).
   \end{multline*}
This completes the proof of the theorem.
\end{proof}

\subsection{Computing rational framed motives}

Recall that $E$-framed motives (of motivic Thom spa\-ces) with rational coefficients can be computed as the motivic complexes $C_*{\bb Q}F^E(Y/V)$.
By~\cite[Example~5.7]{Sch08} $\Tor^{\bb Z[M]}_p(\bb Q,W)=0$ for every tame $M$-module $W$ and all $p\geq 1$.
If we consider $A=\bb Q$, the proof of preceding theorem shows that the complex $\cc K_0\otimes\bb Q$ is quasi-isomorphic
to the presheaf $\wt{\bb Q}F^E(\gmp^{\w q})=\bb Z\otimes_M{\bb Q}F^E(\gmp^{\w q})$ regarded as a complex concentrated in the zeroth degree.
The proof also implies that the morphism of complexes
   $$\lambda:C_*{\bb Q}F^E(\gmp^{\w q})\to C_*\wt{\bb Q}F^E(\gmp^{\w q})$$
is a sectionwise quasi-isomorphism. Replacing $E$ by $E\w Y/V$, $Y\in\Sm_k$ and $V\subset Y$ is open with $Y\setminus V$ smooth, we thus obtain the following
important computation.

\begin{thm}\label{shell}
Let $E$ be a symmetric motivic Thom 
$T$-spectrum with contractible alternating group action. Suppose $Y\in\Sm_k$ and $V\subset Y$ is open 
with $Y\setminus V$ smooth. Then
the natural morphism of complexes of $\bb QF_*(k)$-presheaves
   $$\lambda:C_*{\bb Q}F^E(Y/V)\to C_*\wt{\bb Q}F^E(Y/V)$$
is a quasi-isomorphism. Moreover, the natural morphism of complexes of $\bb QF_*(k)$-sheaves
   $$\lambda:{\bb Q}^E({ q})\to \wt{\bb Q}^E({ q})=\bb Q^E_{\tor}(0,q)$$
is a quasi-isomorphism.
\end{thm}

\begin{cor}[Cancellation]\label{shellcor}
Let $E$ be a symmetric motivic Thom 
$T$-spectrum with contractible alternating group action and the bounding constant $d\leq 1$.
The natural morphism of complexes of sheaves
   $$\gamma:C_*\wt{\bb Q}F^E(Y/V)_{\nis}\to\uhom(\gmp, C_*\wt{\bb Q}F^E(Y/V\w\gmp)_{\nis})$$
is a quasi-isomorphism. In particular, the natural morphism of complexes of sheaves 
   $$\gamma:\wt{\bb Q}^E({ q})\to\uhom(\gmp, \wt{\bb Q}^E({ q}+1))[1]$$
is a quasi-isomorphism.
\end{cor}

\begin{proof}
One has a commutative diagram
   $$\xymatrix{C_*\wt{\bb Q}F^E(Y/V)\ar[r]^(.35)\gamma&\uhom(\gmp, C_*\wt{\bb Q}F^E(Y/V\w\gmp))\\
                       C_*{\bb Q}F^E(Y/V)\ar[r]\ar[u]&\uhom(\gmp, C_*{\bb Q}F^E(Y/V\w\gmp))\ar[u]}$$
It follows from~\cite[Theorem~C]{AGP}, \cite[Proposition~9.12]{GN} and~\cite[Theorem~5.5]{GG26} that the lower arrow is locally a quasi-isomorphism
(note that the spectrum $E\w Y/V$ has the bounding constant $d\leq 1$). The left vertical arrow is a
sectionwise quasi-isomorphism by Theorem~\ref{shell}, and hence so is the right vertical one. We see that
$\gamma$ is locally a quasi-isomorphism.

Next, the proof of~\cite[Theorem~A(2)]{AGP} shows that the map
   $$\uhom(\gmp, C_*\wt{\bb Q}F^E(Y/V\w\gmp))\to\uhom(\gmp, C_*\wt{\bb Q}F^E(Y/V\w\gmp)_{\nis})$$
is locally a quasi-isomorphism. The statement now follows.
\end{proof}

\section{Reconstructing rational stable motivic homotopy theory}\label{rational}

The goal of this section is to prove that permutation-free framed correspondences provide a 
complete description of rational stable motivic homotopy theory $SH(k)_{\bb Q}$ (over any field $k$). Building on 
the constructions of the previous section, we investigate the category of rational 
permutation-free framed correspondences and study the corresponding category of motives.
We prove in this section that this category reconstructs $SH(k)_{\bb Q}$.


Throughout this section $E$ is a symmetric motivic Thom  
$T$-spectrum with contractible alternating group action. 
Recall that for a morphism $f:X\rightarrow Y$ we denote by $\check{C}(f)$ or $\check{C}(X/Y)$ the Cech simplicial object defined by $f$.

\begin{prop}\label{thm4.4}
Let $U\rightarrow X$ be a Nisnevich covering of a scheme $X$. Then for any $n\geq 0$ and any local essentially $k$-smooth
henselian scheme $Y$ the map of pointed simplicial sets $F^E_{n}(Y,\check{C}(U/X))\rightarrow F^E_{n}(Y,X)$ is a weak equivalence.
\end{prop}

\begin{proof}
The proof is like that of~\cite[Theorem~4.4]{Voe2}.
Recall that $E_n=\colim_i V_{n,i}/(V_{n,i}-Z_{n,i})$ (see Section~\ref{mathssmass}). As
local equivalences are closed under filtered colimits, it suffices to prove the statement for spaces of the form $E_n=V/(V-W)$.
The functor $F^E_{n}(Y,-)$ is a coproduct of functors of the form $F_{n}^{E,Z,\phi}(Y,-)$, where $Z$ is a connected closed subset of 
$\mathbb{A}_Y^n$ finite over $Y$, and $\phi\colon (\bb A^n\times Y)^h_Z\to V$ is a map 
such that $Z=\phi^{-1}(W)$. Therefore it is sufficient to show that for any $Z$ and $\phi$ the map
   \begin{equation}\label{eq4.4.1}
    F_{n}^{E,Z,\phi}(Y,\check{C}(U/X))\rightarrow F_{n}^{E,Z,\phi}(Y,X)
   \end{equation}
is a weak equivalence. Note that for any $S\in\Sm_k$ the set $F_n^{E,Z,\phi}(Y,S)$
is isomorphic to the pointed set $\Hom_{\Sm'_k}((\bb A^n\times Y)^h_Z,S)$.
It follows that $F_n^{E,Z,\phi}(Y,-)$ commutes with fiber products and
   \[ F_{n}^{E,Z,\phi}(Y,\check{C}(U/X))=\check{C}(F_{n}^{E,Z,\phi}(Y,U)\rightarrow F_{n}^{E,Z,\phi}(Y,X)). \]
Since $Z$ is connected and finite over the henselian $Y$, then $Z$
is henselian and local, and hence $(\bb A^n\times Y)^h_Z$ is local. Thus the map is 
surjective, and hence~\eqref{eq4.4.1} is a simplicial homotopy equivalence by~\cite[Lemma~14.26.9]{Stack}.
\end{proof}

\begin{cor}\label{thm4.4cor}
Let $f : Y \rightarrow X$ be a Nisnevich covering of $X\in\Sm_k$. Then the following sequence of Nisnevich sheaves is exact:
\[
0 \leftarrow{ (\mathbb{Q}}F_n^E(X)/\Sigma_n)_{\nis} \xleftarrow{\ f_* \ }{ (\mathbb{Q}}F_n^E(Y)/\Sigma_n)_{\nis} \xleftarrow{ (p_2)_* - (p_1)_* }
{ (\mathbb{Q}}F^E_n(Y \times_X Y)/\Sigma_n)_{\nis} \xleftarrow{ (p_{23})_* - (p_{13})_* + (p_{12})_* } \cdots
\]
\end{cor}

\begin{proof}
The functor of coinvariants $(-)/\Sigma_n$ is exact for rational representations. The preceding proposition now implies the claim.
\end{proof}

Recall that a \emph{Nisnevich square\/} is a cartesian diagram of $k$-smooth schemes of the form
\[
\begin{array}{ccc}
W & \longrightarrow & V \\
\downarrow & & \downarrow \lefteqn{p} \\
U & \xrightarrow{\ j \ } & X
\end{array}
\]
where $j : U \hookrightarrow X$ is an open immersion, $p : V \rightarrow X$ is an \'etale morphism, and
the induced map on the reduced complements $p^{-1}(X \setminus U)_{\text{red}} \rightarrow (X \setminus U)_{\text{red}}$ is an isomorphism.

\begin{cor}\label{nissquare}
For any Nisnevich square there is a short exact sequence of Nisnevich sheaves
   $$0\to\wt{ \mathbb{Q}}F^E(W)_{\nis}\to\wt{ \mathbb{Q}}F^E(U)_{\nis}\oplus\wt{ \mathbb{Q}}F^E(V)_{\nis}\to\wt{ \mathbb{Q}}F^E(X)_{\nis}\to 0.$$
\end{cor}

\begin{proof}
The proof of Proposition~\ref{thm4.4} shows that there is a coequalizer of pointed sets
   $$F^E_n(Y,W)\rightrightarrows F_n^E(Y,U)\vee F_n^E(Y,V)\to F_n^E(Y,X),\quad n\geq 0,$$
where $Y$ is a local essentially $k$-smooth henselian scheme. It induces a right exact sequence of Abelian groups
   $$\bb ZF^E_n(Y,W)\to\bb ZF_n^E(Y,U)\oplus\bb ZF_n^E(Y,V)\to\bb ZF_n^E(Y,X)\to 0.$$
Observe that the left arrow of the latter sequence is a monomorphism. 
As the functor of coinvariants $(-)/\Sigma_n$ is exact for rational representations, one gets a short exact sequence
   $$0\to\bb QF^E_n(Y,W)/\Sigma_n\to\bb QF^E_n(Y,U)/\Sigma_n\oplus\bb QF_n^E(Y,V)/\Sigma_n\to\bb QF_n^E(Y,X)/\Sigma_n\to 0.$$
By Lemma~\ref{rakvere} it induces an exact sequence
   $$0\to\wt{ \mathbb{Q}}F^E(Y,W)\to\wt{ \mathbb{Q}}F^E(Y,U)\oplus\wt{ \mathbb{Q}}F^E(Y,V)\to\wt{ \mathbb{Q}}F^E(Y,X)\to 0.$$
Our statement now follows.
\end{proof}

In the remainder of the section $E$ is assumed to be a symmetric Thom commutative ring $T$-spectrum with the bounding 
constant $d\leq 1$ and contractible alternating group action. For instance, $E$ 
is the algebraic cobordism $T$-spectrum
$MGL$ or motivic sphere spectrum $S_T=(S^0,T,T^2,\ldots)$. 
By Theorem~\ref{catcor} we have a symmetric monoidal additive category
$\wt{ \mathbb{Q}}F^E(k)$. The following corollary is proven similarly to~\cite[Corollary~4.5]{Voe2}.

\begin{cor}\label{cor4.5}
Let $F$ be a $\wt{ \mathbb{Q}}F^E(k)$-presheaf. The associated sheaf in the Nisnevich topology has a 
unique structure of a $\wt{ \mathbb{Q}}F^E(k)$-presheaf such that the map $F\rightarrow F_{\nis}$ is a map of $\wt{ \mathbb{Q}}F^E(k)$-presheaves.
\end{cor}

It follows from Corollaries~\ref{nissquare} and~\ref{cor4.5} that $\wt{ \mathbb{Q}}F^E(k)$ is a strict $V$-category of correspondences
in the sense of~\cite{GG19}. For such categories of correspondences we can apply Voevodsky's construction of the triangulated
category of motives (see~\cite{GG19} for details). Namely, we first localize the derived category of $\wt{ \mathbb{Q}}F^E(k)$-sheaves
$D(Sh(\wt{ \mathbb{Q}}F^E(k)))$ with respect to the localizing subcategory $\cc L$
generated by complexes of the form
   $$\cdots\to 0\to\wt{ \mathbb{Q}}F^E(-,X\times\bb A^1)_{\nis}\xrightarrow{pr_X}\wt{ \mathbb{Q}}F^E(-,X)_{\nis}\to 0\to\cdots,\quad X\in\Sm_k.$$
The resulting quotient category $D(Sh(\wt{ \mathbb{Q}}F^E(k)))/\cc L$ is denoted by
$D_{\bb A^1}(Sh(\wt{ \mathbb{Q}}F^E(k)))$.

If we denote by $DM_{E}^{\eff}(k)_{\bb Q}$ the full subcategory of $D(Sh(\wt{ \mathbb{Q}}F^E(k)))$ consisting of
the complexes with strictly $\bb A^1$-invariant cohomology sheaves, then similar to a theorem of Voevodsky~\cite{Voe00}
the composite functor
   $$DM_{E}^{\eff}(k)_{\bb Q}\hookrightarrow D(Sh(\wt{ \mathbb{Q}}F^E(k)))\to D_{\bb A^1}(Sh(\wt{ \mathbb{Q}}F^E(k)))$$
is an equivalence of triangulated categories. Moreover, the functor
   $$C_*:D(Sh(\wt{ \mathbb{Q}}F^E(k)))\to D(Sh(\wt{ \mathbb{Q}}F^E(k))),\quad X\mapsto \Tot(X(-\times\Delta^\bullet)),$$
lands in $DM_{E}^{\eff}(k)_{\bb Q}$. The kernel of $C_*$ is $\cc L$ and $C_*$ is left adjoint to the inclusion
functor
   $$i:DM_{E}^{\eff}(k)_{\bb Q}\hookrightarrow D(Sh(\wt{ \mathbb{Q}}F^E(k)))$$
(see~\cite{Voe00} for details or~\cite[Theorem~3.5]{GP14}). In the special case when $E=S_T$ we will write $DM_{fr}^{\eff}(k)_{\bb Q}$
instead of $DM_{S_T}^{\eff}(k)_{\bb Q}$.

Following~\cite[Section~5]{GG23} we can consider a symmetric monoidal spectral category $\cc O^E$ whose objects are those of $\Sm_k$ with
symmetric spectra of morphisms given by
   $$\cc O^E(X,Y):=(\Fr_0^{E,\Sigma}(X,Y),\Fr_1^{E,\Sigma}(X,Y),\ldots).$$
Following~\cite{GP12} we define the Nisnevich local
model structure on $\Mod\cc O^E$. Its homotopy category, denoted by $D\cc O^{E,\eff}(k)$, is closed symmetric monoidal 
compactly generated triangulated with compact
generators being the representables $\{\cc O^E(-,X)\mid X\in\Sm_k\}$.
The monoidal product $\cc O^E(-,X)\wedge \cc O^E(-,Y)$ in $D\cc O^{E,\eff}(k)$
is isomorphic to $\cc O^E(-,X\times Y)$ --- see~\cite[Theorem~6.11]{GG23}.
The category $D\cc O^{E,\eff}(k)$ plays the same role as the derived category
$D(\textrm{Shv}_{tr}^{\nis}(\Sm_k))$ of cochain complexes of Nisnevich sheaves with transfers.

Recall from~\cite{Voe1} that Voevodsky's category of motives $DM^{\eff}(k)$ is the localisation of the category 
$D(\textrm{Shv}_{tr}^{\nis}(\Sm_k))$
with respect to the family $\{\bb Z_{tr}(-,X\times\bb A^1)\to\bb Z_{tr}(-,X)\mid X\in\Sm_k\}$. It is
equivalent to the full subcategory of $D(\textrm{Shv}_{tr}^{\nis}(\Sm_k))$ consisting of
cochain complexes with homotopy invariant cohomology sheaves~\cite{Voe1}.
Likewise, we can localize $\mathsf{Mod}\cc O^E$ equipped with the Nisnevich local model structure with respect to the maps 
$\{\cc O^E(-,X\times\bb A^1)\to\cc O^E(-,X)\mid X\in\Sm_k\}$.
Denote by $D\cc O^{E,\eff}_{\mot}(k)$ its homotopy category. By~\cite[Theorem~6.12]{GG23} the
category $D\cc O^{E,\eff}_{\mot}(k)$ is equivalent to the full triangulated subcategory, denoted by $DE^{\eff}(k)$,
of $D\cc O^{E,\eff}(k)$ consisting of modules with homotopy invariant sheaves of stable homotopy groups.
The inclusion $DE^{\eff}(k)\to D\cc O^{E,\eff}(k)$ has a left adjoint
$C_*$ taking a module $M\in D\cc O^{E,\eff}(k)$ to its Suslin complex $C_*(M)$. Moreover,
there is a triangulated equivalence of compactly generated triangulated categories
   $$DE^{\eff}(k)\simeq\Mod^{\eff}_{SH(k)}E.$$
Here $\Mod_{SH(k)}E$ denotes the homotopy category of
motivic $E$-module spectra, obtained by inverting the
morphisms that are stable motivic equivalences on the
underlying spectra. The category $\Mod^{\eff}_{SH(k)}E$
is the full triangulated subcategory of $\Mod_{SH(k)}E$
consisting of modules whose underlying $T$-spectra
are effective~\cite[Corollary~6.13]{GG23}.

There is a natural map of spectral categories
   $$\alpha:\cc O^E\to H(\wt{\pi}_0(\cc O^E)),$$
where $H(\wt{\pi}_0(\cc O^E))$ is the Eilenberg--Mac~Lane spectral category associated with the ringoid
$\wt{\pi}_0(\cc O^E)$. It follows from Corollary~\ref{naivevstruecor} that $\wt{\pi}_0(\cc O^E)$ equals $\wt{\bb Z}F^E(k)$.
Denote by $HE_{\bb Q}:=H(\wt{\bb Q}F^E(k))$. The map $\alpha$ induces a pair of adjoint functors
   \begin{equation}\label{pair}
    F:\Mod\cc O^E\rightleftarrows\Mod HE_{\bb Q}:U
   \end{equation}
Note that $F$ is strong monoidal.

It follows from Corollary~\ref{nissquare} that the spectral category $HE_{\bb Q}$ is Nisnevich excisive in the sense of~\cite{GP12}.
Therefore we can define the motivic model structure on $\Mod HE_{\bb Q}$ whose homotopy category will be denoted by $DHE_{\mot}^{\eff}(k)_{\bb Q}$.
The pair $(F,U)$~\eqref{pair} is a Quillen pair with respect to the motivic model structure. Therefore one has a pair
of adjoint functors
   $$F:D\cc O_{\mot}^{E,\eff}(k)\rightleftarrows DHE_{\mot}^{\eff}(k)_{\bb Q}:U$$
Moreover, $F$ is strict monoidal.
By~\cite[Section~6]{GP12} there is a pair of triangulated equivalences between triangulated categories
   $$L:DHE_{\mot}^{\eff}(k)_{\bb Q}\rightleftarrows D_{\bb A^1}(Sh(\wt{ \mathbb{Q}}F^E(k))):R$$
Moreover, $L$ is strict monoidal.

We are now in a position to prove the following theorem.

\begin{thm}\label{reconstreff}
Let $k$ be any field. There is a symmetric monoidal triangulated equivalence 
   $$DE^{\eff}(k)\otimes\bb Q\simeq DM_E^{\eff}(k)_{\bb Q},$$
where $DE^{\eff}(k)\otimes\bb Q$ is defined via rational Bousfield localization.
In particular, $\Mod^{\eff}_{SH(k)}E\otimes\bb Q$ is equivalent to $DM_E^{\eff}(k)_{\bb Q}$ and 
$SH^{\eff}(k)\otimes\bb Q$ is equivalent to $DM_{fr}^{\eff}(k)_{\bb Q}$.
\end{thm}

\begin{proof}
Due to arguments above all we need to prove is that the pair~\eqref{pair} induces an equivalence of 
compactly generated triangulated categories
   $$F:D\cc O_{\mot}^{E,\eff}(k)\otimes\bb Q\rightleftarrows DHE_{\mot}^{\eff}(k)_{\bb Q}:U$$
By~\cite[Theorem~6.7]{GG23} $D\cc O_{\mot}^{E,\eff}(k)$ is compactly generated triangulated with compact
generators being the symmetric $E$-framed motives $\{M^\Sigma_E(X)\mid X\in\Sm_k\}$ (see~\cite[Theorem~6.3]{GG26}
for imperfect fields).
By~\cite[Lemma~6.5]{GG23} the natural map $M_E(X)\to M^\Sigma_E(X)$ is a stable equivalence. Thus
one has a zigzag of sectionwise stable equivalences
   $$LM_E(X)\otimes\bb Q\xleftarrow{\sim}M_E(X)\otimes\bb Q\lra{\sim}M^\Sigma_E(X)\otimes\bb Q,$$ 
where $LM_E(X)$ is the linear $E$-framed motive in the sense of~\cite[Section~9]{GN}. Its presheaves
of stable homotopy groups are computed as the presheaves of homology groups of the complex $C_*\bb ZF^E(X)$.
By Theorem~\ref{shell} $C_*\bb QF^E(X)$ is sectionwise quasi-isomorphic to $C_*\wt{\bb Q}F^E(X)=C_*\wt{\bb Z}F^E(X)\otimes\bb Q$.

Likewise, $DHE_{\mot}^{\eff}(k)_{\bb Q}$ is compactly generated triangulated with compact
generators $\{M_{HE_{\bb Q}}(X)\mid X\in\Sm_k\}$, where $M_{HE_{\bb Q}}(X)=C_*(HE_{\bb Q}(-,X))$. The
presheaves of stable homotopy groups of $M_{HE_{\bb Q}}(X)$ are computed as the presheaves of homology groups of the complex
$C_*\wt{\bb Q}F^E(X)$.

It follows that $F$ takes compact generators to compact generators. Moreover, $F$ induces isomorphisms between
$D\cc O_{\mot}^{E,\eff}(k)\otimes\bb Q(M^\Sigma_E(X),M^\Sigma_E(Y)[i])$ and
$DHE_{\mot}^{\eff}(k)_{\bb Q}(M_{HE_{\bb Q}}(X),M_{HE_{\bb Q}}(Y)[i])$. By~\cite[Lemma~4.8]{GJ} this is enough to conclude that 
$F$ is a triangulated equivalence.
\end{proof}

We can stabilize our constructions in the $\gmp$-direction.
Denote by $-\boxtimes\gmp$ the endofunctor 
   $$\cc X\in\mathsf{Mod}\cc O^E\mapsto\cc X(\bb G_m^{\wedge 1})=\cc X\wedge_{\cc O^E}\cc O^E(-,\gmp).$$
Following Hovey~\cite[Section~8]{H}, we consider the stable motivic model structure on $\gmp$-symmetric spectra 
$Sp^\Sigma(\mathsf{Mod}\cc O^E,\gmp)$ (we start with the motivic stable model structure
on $\mathsf{Mod}\cc O^E$). Its homotopy category is denoted by $\mathsf{SH}_{S^1,\bb G_m}\cc O^E$.

By~\cite[Corollary~6.9]{GG23} there is
an equivalence of compactly generated triangulated categories
   $$\mathsf{SH}_{S^1,\bb G_m}\cc O^E\simeq\Mod_{SH(k)}E.$$
This equivalence is also true with $1/p$-coefficients over imperfect fields --- 
see~\cite[Theorem~6.5]{GG26}. Note that $\mathsf{SH}_{S^1,\bb G_m}\cc O^E$
is equivalent to the category $DE(k)$ obtained from $DE^{\eff}(k)$ by stabilising the $\gmp$-direction. 
We define the triangulated category $DM_E(k)_{\bb Q}$ in a similar fashion. It is obtained from $DM^{\eff}_E(k)_{\bb Q}$
by stabilising in the $\gmp$-direction. In the special case when $E=S_T$ we will write $DM_{fr}(k)_{\bb Q}$
instead of $DM_{S_T}(k)_{\bb Q}$.

We finish the section by reconstructing stable motivic homotopy
theory with rational coefficients out of permutation free rational framed correspondences.

\begin{thm}[Reconstruction]\label{reconstr}
Let $k$ be any field. There is a symmetric monoidal triangulated 
equivalence between triangulated categories with rational coefficients
   $$DE(k)\otimes\bb Q\simeq DM_E(k)_{\bb Q}.$$
In particular, $\Mod_{SH(k)}E\otimes\bb Q$ is equivalent to $DM_E(k)_{\bb Q}$ and 
$SH(k)\otimes\bb Q$ is equivalent to $DM_{fr}(k)_{\bb Q}$.
\end{thm}

\begin{proof}
The equivalence $G:DE(k)\otimes\bb Q\to DM_E(k)_{\bb Q}$ is induced by the triangulated equivalence of
Theorem~\ref{reconstreff}. In more detail, similarly to the proof of Theorem~\ref{reconstreff} the fact that $G$
is an equivalence reduces to comparing Hom-sets between compact generators in both categories.
It follows from Corollary~\ref{shellcor} that the strict $V$-category of correspondences $\wt{\bb Q}F^E(k)=\wt{\bb Z}F^E(k)\otimes\bb Q$
satisfies the cancellation property in the sense of~\cite[Definition~3.5]{GG19}. It follows from~\cite[Lemma~3.6]{GG19} that
$DM_{E}(k)_{\bb Q}$ is compactly generated by motivically fibrant $\gmp$-complexes 
   \begin{equation}\label{almeria}
    (\uhom(\gmp^{\w n},C_*\wt{\bb Q}F^E(X)_f),\uhom(\gmp^{\w n},C_*\wt{\bb Q}F^E(X_+\w\gmp)_f),\ldots),\quad n\geq 0,
   \end{equation}
where the subscript ``$f$" refers to the injective resolution in the category of $\wt{\bb Q}F^E(k)$-sheaves.

The category $DE(k)$ is compactly generated by the following motivically fibrant bispectra
   $$(\Omega_{\gmpn}M_E^\Sigma(X)_f,\Omega_{\gmpn}M_E^\Sigma(X_+\wedge\gmp)_f,\Omega_{\gmpn} M_E^\Sigma(X_+\wedge\bb G^{\wedge 2})_f,\ldots),\quad n\geq 0,$$
where ``$f$'' refers to level local
fibrant replacements of motivic $S^1$-spectra. The proof of Theorem~\ref{reconstreff} implies the presheaves of 
stable homotopy groups~\eqref{bigraded} with rational coefficients of this bispectrum are isomorphic to those of~\eqref{almeria}.
By~\cite[Lemma~4.8]{GJ} we can conclude that $G$ is an equivalence.
\end{proof}

\end{document}